\documentclass[a4paper,11pt]{article}
\usepackage{mathdots}
\usepackage{amsmath,amsfonts,enumerate,amssymb,amsthm,color}
\usepackage{graphicx}
\usepackage[margin=2cm]{geometry}
\newtheorem{lemma}{Lemma}[section]

\rmfamily
\newtheorem{theorem}[lemma]{Theorem}

\theoremstyle{definition}
\newtheorem{remark}[lemma]{Remark}

\newtheorem{example}[lemma]{Example}
\newtheorem{definition}[lemma]{Definition}

\newcommand{\comment}[1]{}
\newcommand{\sbm}[1]{\begin{bmatrix} #1 \end{bmatrix}}

\DeclareMathOperator{\codim}{codim}
\DeclareMathOperator{\Cov}{Cov}
\DeclareMathOperator{\Inn}{Inn}
\DeclareMathOperator{\N}{N}

\DeclareMathOperator{\Aut}{Aut}
\newcommand{\Dih}{{\rm{Dih}}}
\DeclareMathOperator{\beg}{beg}
\DeclareMathOperator{\inv}{inv}
\DeclareMathOperator{\V}{V}

\newcommand{\cI}{{\mathcal{I}}}

\newcommand{\generat}{{\rm \hbox{gen}}}

\newcommand{\DC}{{\rm C}_n^{(2)}}

\DeclareMathOperator{\D}{D}

\newcommand{\la}{\langle}
\newcommand{\ra}{\rangle}
\newcommand{\un}{\underline}

\renewcommand{\mod}{\hbox{{\rm mod}}}

\newcommand{\id}{{\rm id}}

\newcommand{\ZZ}{\mathbb{Z}}
\newcommand{\NN}{\mathbb{N}}

\newcommand{\GL}{{\rm GL}}

\newcommand{\normal}{\triangleleft}

\newcommand{\gen}[1]{\langle #1 \rangle}
\newcommand{\ad}{{\rm ad}}
\newcommand{\p}{\wp}
\newcommand{\Orb}{{\rm Orb}}
\newcommand{\HG}{{\rm H}}
\newcommand{\expn}{{\rm exp}}
\newcommand{\Expn}{{\rm Exp}}
\newcommand{\ind}{{\rm ind}}
\renewcommand{\deg}{{\rm deg}}

\begin{document}


\begin{center}
{\bf\Large Tetravalent Vertex- and Edge-Transitive Graphs Over Doubled Cycles} \\ [+4ex]
Bo\v stjan Kuzman{\small$^{a,d,}$}\footnotemark,
Aleksander Malni\v c{\small$^{a,c,d,}$}\footnotemark,
Primo\v z Poto\v cnik{\small$^{b,d,}$}
\footnotemark
\\ [+2ex]
{\it \small
$^a$University of Ljubljana, UL PEF, Kardeljeva pl.~16, 1000 Ljubljana, Slovenia\\
$^b$University of Ljubljana, UL FMF, Jadranska 19, 1000 Ljubljana, Slovenia\\
$^c$University of Primorska, UP IAM, Muzejski trg 2, 6000 Koper, Slovenia\\
$^d$IMFM, Jadranska 19, 1000 Ljubljana, Slovenia\\
}
\end{center}

\addtocounter{footnote}{-2}
\footnotetext{This work is supported in part by the Slovenian Research Agency (research program P1-0285 and research project  J1-7051).}
\addtocounter{footnote}{1}
\footnotetext{This work is supported in part by the Slovenian Research Agency (research program P1-0285 and research projects N1-0032, N1-0038, N1-0062, J1-6720, and J1-7051).}
\addtocounter{footnote}{1}
\footnotetext{This work is supported in part by the Slovenian Research Agency (research program P1-0294).

\smallskip
\quad
Email addresses:
bostjan.kuzman@gmail.com,\ 
aleksander.malnic@guest.arnes.si,\ 
primoz.potocnik@fmf.uni-lj.si

}


\begin{abstract}
In order to complete (and generalize) results of  Gardiner and Praeger on 4-valent symmetric graphs (European J. Combin, 15 (1994)) we apply the method of lifting
automorphisms  in the context of  elementary-abelian covering projections.  In particular, the vertex- and edge-transitive graphs whose quotient 
by a normal $p$-elementary abelian group of automorphisms,  for $p$ an odd prime,
is a cycle, are described in terms of cyclic and negacyclic codes. Specifically, the symmetry properties of such graphs are derived from certain properties of the generating polynomials of cyclic and negacyclic codes, that is, from divisors of $x^n \pm 1 \in \ZZ_p[x]$. As an application, a short and unified description of  resolved and unresolved cases of  Gardiner and Praeger are given.
 \end{abstract}



\bigskip

\section{Introduction}
\label{sec:intro}

A classical topic in the algebraic graph theory is the classification of interesting families of graphs with specific symmetry properties, arc-transitive in particular.  Following the pioneering work of Tutte \cite{Tut}, the class of cubic arc-transitive graphs
is certainly the  most widely studied while the $4$-valent case is considerably less well understood,
with some important results appearing recently \cite{PSV,PWarchive,genlost}.  

A systematic approach to these graphs was initiated by Gardiner and Praeger \cite{GP0, GP1}. Let us briefly explain their approach (see Section~\ref{sec:prelim} for the definitions of all concepts used in this introductory section).
Let $\Gamma$ be a connected finite simple 4-valent $G$-arc-transitive graph. If $G$ contains no normal abelian subgroups, then its structure is rather restricted: 
 it is isomorphic to a group $G^*$ satisfying $\Inn R\leq G^*\leq \Aut R$, where $R$ is the direct product of the minimal normal subgroups of $G$ (such groups $G$ are called  semisimple, see \cite[p.~89]{Rob}). This structural information on $G$ yields severe restrictions on the graph $\Gamma$; for example, it was shown in \cite{PSV} that the order of the vertex stabilizer $G_v$ is bounded by a sub-linear function of the order $n$ of the graph $\Gamma$ --  which is in strong contrast with the general case where $|G_v|$ cannot be bounded by any sub-exponential function of $n$.

While the case where $G$ is semisimple
is interesting from the group-theoretical point of view,
 the case where $G$ contains a nontrivial normal abelian subgroup often yields graphs with richer combinatorial structure.
Henceforth, let $N$ be an abelian minimal normal subgroup of $G$. Then $N$ is isomorphic  to $\ZZ_p^r$ for some prime $p$  and a positive integer $r$. 
Consider the quotient graph $\Gamma_N$ (where the quotient is defined as in \cite{GP0}, see also Section~\ref{sec:prelim}). It is easy to see that $\Gamma_N$ is  isomorphic either to $K_1$, $K_2$, a cycle $C_n$, $n\geq 3$, or to  a $4$-valent simple connected $G/N$-arc-transitive graph \cite[Theorem 1.1]{GP0}. The latter case naturally leads to an inductive reduction. If $\Gamma_N$ is isomorphic to $K_1$ or $K_2$, the possibilities for the graphs $\Gamma$ were completely determined in \cite[Theorems 1.2 and 1.3]{GP0}.
The remaining case where $\Gamma_N$ isomorphic to $C_n$ for some $n\geq 3$ is considered in \cite[Theorems 1.1]{GP1}, which we here state with minor notational changes.

\begin{theorem}{\rm \cite[Theorem 1.1]{GP1}}
\label{GaP1}
Let $\Gamma$ be a connected, $H$-symmetric, $4$-valent graph, and let $N$ be a minimal normal $p$-subgroup of $H$ 
with orbits of size $p^r$ for some prime $p$. Let $M$ denote the kernel of the action of $H$ on $N$-orbits. Suppose that the simple quotient $\Gamma$ with respect to $N$ is a cycle $C_n$ of length $n\geq 3$. Then $r\leq n$ and one of the following holds (refer to \cite{GP1} for the definitions of the respective graphs):
\begin{enumerate}[(a)]
\item $p=2$, the centralizer $C_M(N)$ does not act semiregularly on $V(\Gamma)$, and $\Gamma=C(2;n,r)$.
\item $p$ is odd, $C_M(N)$ acts semiregularly on $V(\Gamma)$, and the stabilizer $M_v$ of a vertex $v$
is a nontrivial elementary-abelian $2$-group of order dividing $2^r$. In the extremal case in which 
$|M_v|=2^r$, we have $\Gamma=C^{\pm 1}(p;rt,r)$ or $\Gamma=C^{\pm\theta}(p;2rt,r)$ for some $t\geq 1$.
\end{enumerate}
\end{theorem}

Note that the classification problem in the above theorem is completely resolved for $p=2$, but remains largely unsettled for $p$ odd.  Some additional information about this case was provided in \cite[Theorem 1.2]{GP1}.
In this paper we resolve part (b) of Theorem~\ref{GaP1} completely and generalize it in the sense that we allow the group  $H$ to be vertex- and edge-, but not necessarily arc-transitive.
Before stating our main result (see Theorem~\ref{thm:main}) we  present a construction of certain $4$-valent graphs arising from polynomials that are related to linear codes.

\begin{definition}
\label{def:matrix}
Let $p$ be an odd prime, $n\in \NN$ a positive integer, $\epsilon\in \{0,1\}$, and
let $\Delta_{n,\epsilon}(x)$ denote the polynomial
$\Delta_{n,\epsilon}(x)=x^n-(-1)^\epsilon\in\ZZ_p[x]$ over the prime field $\ZZ_p$.
For any proper divisor $g(x)$ of $\Delta_{n,\epsilon}(x)$, say
$$g(x)=\alpha_0+\alpha_1x+\ldots + \alpha_m x^m\in\ZZ_p[x],$$ 
set $r=n-m$,  and let 
$$
M_{g(x)} =\left[\begin{array}{cccccccccc} 
\alpha_0 &  \ldots & \alpha_m & 0 & \ldots &&&\cdots & 0 \\
0  & \alpha_0 & \ldots & \alpha_m & \ddots& &&&\vdots \\
\vdots
& \ddots&  \ddots  &&\ddots&& \\ 

\\
&&&&\ddots&&\ddots&\ddots&\vdots\\
\vdots  && && \ddots& \alpha_0 &\ldots& \alpha_m&0
\\

0  &\cdots&&& \cdots& 0& \alpha_0&\ldots& \alpha_m
\end{array}\right] \in\ZZ_p^{r\times n}
$$ 
be the $r \times n$  matrix associated with the given polynomial $g(x)$. 
The $4$-valent graph $\Gamma_{g(x)}$ of order $n\cdot p^r$ is now defined by the vertex set $\ZZ_p^{r}\times \ZZ_n$ and adjacency relations
$(v,j)\sim (v\pm u_{j+1},j+1)$ for $j \in \ZZ_n$,  where $u_{j+1}\in \ZZ_p^r$ corresponds to the $(j+1)$-th column of the matrix $M_{g(x)}$. 
 \end{definition}

As will be proved in Theorem~\ref{thm:maxG-kernel}, these graphs are always vertex- and edge-transitive. In order to characterize those that are also arc-transitive we need to introduce some further concepts.

\begin{definition}
\label{def:reflexible}
A nonzero polynomial  $f(x)=\alpha_0+\alpha_1x+\ldots + \alpha_m x^m \in \ZZ_p[x]$   is {\em reflexible} whenever there exists some $\lambda \in \ZZ_p^*$ such that either $ \lambda \alpha_{m-i} = \alpha_i$ for all $i\in \{0,\ldots, m\}$ (referred to as {\em type-1 reflexibility}), or $ \lambda \alpha_{m-i} = (-1)^i\alpha_i$ for all $i\in \{0,\ldots, m\}$ (referred to as {\em type-2 reflexibility}). 
\end{definition}

\begin{definition}
\label{def:assoc-poly}
For a polynomial $f(x) \in \ZZ_p[x]$,  let $\expn(f(x))$ denote  the set of all positive integers $k \in \NN$ such that $f(x)$ is a polynomial in $x^k$. For 
 a polynomial of degree $n \geq 1$ this set is bounded by   $d = \Expn(f(x)) = \max\{k\ |\ k \in \expn(f(x))\} \leq n$;  the uniquely defined polynomial $f_d(x) \in \ZZ_p[x]$
such that $f(x) = f_d(x^d)$ is referred to as the {\em polynomial associated with $f(x)$}. 
These concepts will be mainly applied when considering proper divisors of the polynomial 
$\Delta_{n,\epsilon}(x)$ (viewed as elements in the quotient ring
$\ZZ_p[x]/(\Delta_{n,\epsilon}$). In this context we define $\Expn(f(x))$ of the constant polynomial $f(x) = 1$ 
 to be $\Expn(f(x)) = n$,  and its  associated polynomial is the constant polynomial $f_n(x) =1$.  
\end{definition}

\begin{definition}
\label{def:weak}
We call  a nonconstant polynomial $f(x) \in \ZZ_p[x]$ \emph{weakly reflexible} whenever its associated polynomial  is reflexible (the constant polynomial $f(x) =1$, considered as a divisor of $\Delta_{n,\epsilon}(x)$, is reflexible and weakly reflexible). A weakly reflexible proper divisor $g(x)$ of $\Delta_{n,\epsilon}(x)$ is by defnition  \emph{maximal} whenever $g(x)$ is not a proper divisor of another weakly reflexible proper divisor of $\Delta_{n,\epsilon}(x)$. 
\end{definition}

Note that certain polynomials enjoy both types of reflexibility (in fact, in such cases there is no distinction between the two types), see Lemma~\ref{lem:refl-unique}.   
Also,  a reflexible polynomial  is  weakly reflexible (but the converse might not hold, see Lemma~\ref{lem:f-refl-fd-refl}). Hence  if a maximal proper divisor of $\Delta_{n,\epsilon}(x)$ is reflexible, then it is a maximal weakly reflexible divisor.

\begin{example}
Let $g(x)=3+4x^2+2x^4+x^6 \in\ZZ_5[x]$. Then $g(x)$ divides $\Delta_{8,0}(x)=x^{8}-1\in\ZZ_5[x]$, but is not reflexible since $\lambda\alpha_6=\alpha_0$ would imply $\lambda=3$ while $3\alpha_4\ne \alpha_2$. However,  its associated polynomial  $g_2(x)=3+4x+2x^2+x^3$ is reflexible (of type 2) with $\lambda=3$, and so $g(x)$ is weakly reflexible. 
The respective matrix 
$$M_{g(x)}=\begin{bmatrix}
3&0&4&0&2&0&1&0\\
0&3&0&4&0&2&0&1
\end{bmatrix}\in\ZZ_5^{2\times 8}
$$
yields the $4$-valent graph $\Gamma_{g(x)}$ of order $8\cdot 5^2=200$. This graph is arc-transitive. Computation in  Magma shows that
it is isomorphic to the graph $C4[200,22]$ from the Census of symmetric 4-valent graphs by Poto\v cnik and Wilson \cite{PW census}.


\end{example}

\begin{example}
The divisor $g(x)=1+x+x^2+2x^3+x^5$ of $\Delta_{8,0}(x)\in\ZZ_3[x]$ is not weakly reflexible since $g_1(x) = g(x)$ is not reflexible. The respective graph of order $8\cdot 3^3=216$
is isomorphic to  $C4[216,33]$, and is vertex- and edge- but is not arc-transitive \cite{PW census}. 
\end{example}

Our main theorem is now stated as follows.
\begin{theorem}
\label{thm:main}
Let $\Gamma$ be a finite connected simple $4$-valent graph, and let $H$ be a vertex- and edge-transitive subgroup  of $\Aut(\Gamma)$. Further, let $N$ be a minimal normal subgroup of $H$ isomorphic to $\ZZ_p^r$ for some odd prime $p$ and positive integer $r$, and suppose that the simple quotient  graph $\Gamma_N$ is a cycle of length $n\geq 3$.
 
Then $r\leq n$, the group  $N$ acts semiregularly on $V(\Gamma)$, and $\Gamma$ is isomorphic to a graph $\Gamma_{g(x)}$ for some polynomial $g(x)\in\ZZ_p[x]$ of degree $m=n-r$ dividing $\Delta_{n,\epsilon}(x)$  for some $\epsilon \in \{0,1\}$. Let 
$d = \Expn(g(x))$. Then $d$ is a divisor of both $n$ and $r$.   

Moreover, the normalizer  $\N_{\Aut(\Gamma)}(N)$
is arc transitive if and only if $g(x)$ is weakly reflexible. 
If $g(x)$ is not weakly reflexible, then $g_d(x)$ is a maximal divisor of  $\Delta_{n/d,\epsilon}(x)$, and the vertex stabilizer of  $\N_{\Aut(\Gamma)}(N)$  is isomorphic  to  $\ZZ_2^d$. If $g(x)$ is weakly reflexible, then  the associated polynomial $g_d(x)$ is a maximal weakly reflexible divisor of $\Delta_{n/d,\epsilon}(x)$, and the vertex stabilizer of  $\N_{\Aut(\Gamma)}(N)$  is isomorphic  to  $\ZZ_2 \ltimes \ZZ_2^d$. 
\end{theorem}

\begin{figure}
\begin{center}
\includegraphics[width=10cm]{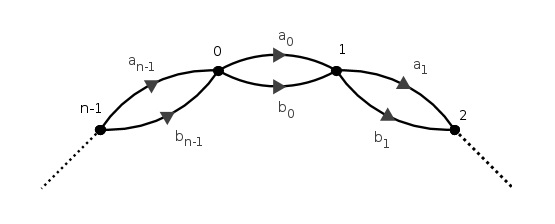}
\caption{Doubled cycle $\DC$ with arcs $a_i$, $b_i$.}
\label{fig:dcikli}
\end{center}
\end{figure}

The method that enabled us to resolve the unsettled cases of \cite[Theorems 1.1 and 1.2]{GP1} is that of lifting groups along elementary-abelian covering projections of graphs, as developed in \cite{MMP}. In particular, Theorem~\ref{thm:main} is derived from  Theorems~\ref{thm:min-const} and \ref{thm:maxG-kernel}  stated in the language of graph covers. To this end we need to introduce some further notation.
By $\DC$ we denote the {\em doubled cycle of length $n$}, that is, the multigraph with the vertex-set $\ZZ_n$ and two distinct edges between vertices $i$ and $i+1$, $i \in \ZZ_n$, with $a_i,b_i$ denoting the respective arcs from $i$ to $i+1$ (see Figure~\ref{fig:dcikli}).
Finally, we denote by $\wp_{g(x)}$ the covering projection ${\wp}_{g(x)}\colon \Gamma_{g(x)}\to \DC$ that maps an arc $((v,j),(v+ u_{j+1},j+1))$ to $a_j$ and an arc $((v,j),(v- u_{j+1},j+1))$ to $b_j$. The terminology related to graph covers appearing in the following two theorems is explained in Subsection~\ref{subsec:covproj}, and in more detail in \cite{MMP}.

\begin{theorem}
\label{thm:min-const}
Let $G$ be a vertex- and edge-transitive group of automorphisms of $\DC$,  and let  $p$ be an odd prime. Then any minimal  $G$-admissible $p$-elementary abelian covering projection $\wp\colon \Gamma \to \DC$, where $\Gamma$ is connected and simple, is isomorphic to a covering projection ${\wp}_{g(x)}\colon \Gamma_{g(x)}\to \DC$  for some proper divisor $g(x)$ of $\Delta_{n,\epsilon} \in \ZZ_p[x]$ and some $\epsilon \in\{0,1\}$.
\end{theorem}

\begin{theorem} 
\label{thm:maxG-kernel}
Let $G$ be the maximal group that lifts along the covering projection $\wp_{g(x)}\colon \Gamma_{g(x)} \to \DC$, where  
 $g(x)$ is a proper divisor of $\Delta_{n,\epsilon}(x) \in \ZZ_p[x]$ and $\epsilon \in\{0,1\}$.  

Let $d = \Expn(g(x))$. Then the kernel of the action of $G$ on the vertex set of $\DC$  is isomorphic to  $\ZZ_2^d$, and $G$ is vertex- and edge-transitive. Moreover, $G$ is arc-transitive if and only if  $g(x)$ is weakly reflexible. In particular, $G = \Aut(\DC)$ if and only if $g(x) =1$.

Additionally, if $g(x)$ is not weakly reflexible, then $\wp_{g(x)}$  is a minimal $G$-admissible covering if and only if the associated polynomial $g_d(x)$ is a maximal  divisor of  $\Delta_{n/d,\epsilon}(x)$.
If $g(x)$ is weakly reflexible, then $\wp_{g(x)}$  is a minimal $G$-admissible covering if and only if $g_d(x)$ is a maximal weakly reflexible divisor of 
 $\Delta_{n/d,\epsilon}(x)$. In particular, the covering arising from $g(x) = 1$ is minimal
$(\Aut(\DC)$-admissible.
\end{theorem}

\begin{example}
Let $n=3$, $p=7$ and $\epsilon=0$. In the table below we list all proper divisors of $\Delta_{3,0}(x)=x^3-1\in\ZZ_7[x]$, the corresponding graphs and their symmetries.
For each polynomial divisor $g(x)$, the columns wr and mwr denote whether $g(x)$ is weakly reflexible or maximal weakly reflexible, respectively. We identify the graph $\Gamma=\Gamma_{g(x)}$ by its code in the Census of edge-transitive 4-valent graphs \cite{PW census}, and denote its symmetry type by AT (arc-transitive) and HT (half-arc-transitive, that is, vertex- and edge- but not arc-transitive).  By $G$ we denote the maximal subgroup of $\Aut\DC$ (up to conjugation) that lifts along $\wp_{g(x)}$, and by $\tilde{G}$ its lifted group.
Observe that in rows 2 and 3, the graphs are arc-transitive even though the polynomial is not weakly reflexible. Indeed, it turns out in this case that the lifted group $\tilde{G}$ is much smaller than the full automorphism group.(For the notation in the $G$-column see 
Subsection~\ref{subsec:AutDC}.)
\begin{center}
\begin{tabular}{c|lcclrllrr}
Row& $g(x)$ & wr & mwr & graph $\Gamma$ &$|V(\Gamma)|$& sym & $G$ & $|\tilde{G}|$ & $ |\Aut \Gamma|$ \\ \hline
1& $x^2+x+1$ & Y&Y&$C4[21,1]$&21& AT & $\gen{\tau_{\ZZ_3},\rho,\sigma}$ & 84&84\\
2& $x^2+4x+2$ & N&N&$C4[21,2]$&21& AT  & $\gen{\tau_{\ZZ_3},\rho}$& 42 &336\\
3& $x^2+2x+4$& N &N&$C4[21,2]$&21& AT & $\gen{\tau_{\ZZ_3},\rho}$&42 &336\\
4& $x+5$ & N&N&$C4[147,5]$&147& HT&   $\gen{\tau_{\ZZ_3},\rho}$&294&294 \\
5& $x+3$ &N&N&$C4[147,5]$&147& HT&  $\gen{\tau_{\ZZ_3},\rho}$ &294&294\\
6& $ x+6$&Y&Y&$C4[147,6]$&147& AT&   $\gen{\tau_{\ZZ_3},\rho,\sigma}$& 588&588\\
7& $1$&Y&N&-- &1029& AT & $\gen{\tau_{\ZZ_3},\rho,\sigma}$ & 8232 & 16464\\
\end{tabular}
\end{center}
\end{example}

A brief account of the theory of lifting groups along covering projections is given in Section~\ref{sec:prelim}. In Section~\ref{sec:Proof-4-5-3}, these results are applied to prove Theorems~\ref{thm:min-const} and ~\ref{thm:maxG-kernel}, and to complete the proof of Theorem~\ref{thm:main}. Finally, two further examples of our construction corresponding  to the results of \cite{GP1} are described in  Section~\ref{sec:GardPra}.

\section{Preliminaries}
\label{sec:prelim}

Let $\Gamma$ be a graph and $N\leq\Aut(\Gamma)$ some automorphism group. Following \cite{GP0}, we let $\Gamma_N$ be the simple graph whose vertices are the orbits $v^N$ of $N$ on $V(\Gamma)$, and two distinct orbits $v^N$, $u^N$ are adjacent whenever there is an edge in $\Gamma$ between them: $v'\sim u'$ for some $v'\in v^N$, $u'\in u^N$. (Note that we shall distinguish between the \emph{simple quotient} $\Gamma_N$ and the \emph{quotient} $\Gamma/N$ as defined in Section~\ref{subsec:covproj}, which can be a nonsimple graph.)

\subsection{Graphs and their automorphisms}

\label{subsec:graphs}

Although we are mainly interested in simple graphs, it is convenient to allow loops, multiple edges and semiedges for detailed discussion of covering projections. Following~\cite{MMP}, we define a {\em graph} to be an ordered $4$-tuple $(D,V, \beg,\inv)$ where
$D$ and $V \neq \emptyset$ are disjoint finite sets of {\em darts} and {\em vertices}, respectively, $\beg\colon D \to V$ is a mapping
which assigns to each dart $x$ its {\em initial vertex} $\beg x$, and $\inv\colon D \to D$ is an involution which interchanges
every dart $x$ with its {\em inverse dart}, also denoted by $x^{-1}$. If $\Gamma$ is a graph, then we let $\D(\Gamma)$,  $\V(\Gamma)$, $\beg_\Gamma$ and $\inv_\Gamma$ denote its dart-set, its vertex-set, its $\beg$ function and its $\inv$ function, respectively.

The orbits of $\inv$ are called {\em edges}. The edge containing a dart $x$ is called a {\em semiedge} if $x^{-1} = x$, a {\em loop} if $x^{-1} \neq x$ while  $\beg x^{-1} = \beg x$, and  is called  a {\em link} if $x^{-1} \neq x$ and $\beg x^{-1} \neq \beg x$. The {\em endvertices of an edge} are the initial vertices of the darts contained in the edge. Two links are {\em parallel} if they have the same endvertices. A graph is called \emph{simple} if it has no semiedges, loops or parallel links. The cardinality $|V|$ of $V$ is called the {\em order} of $\Gamma$. The {\em neighborhood} of a vertex $v$, denoted $\Gamma(v)$,  is the set consisting of all the darts $x$ with $\beg x=v$, and the cardinality of $\Gamma(v)$ is called the {\em valence} of $v$. A graph is {\em tetravalent} if all of its vertices have valence $4$.

Note that in a simple graph each edge is uniquely determined by its endvertices, implying that simple graphs can be given in the usual manner, by specifying the vertex set $V$ and the set $E$ of unordered pairs of vertices corresponding to the edges of the graph. Moreover, a dart $x$ of a simple graph is uniquely determined by the ordered pair $(\beg x, \beg x^{-1})$. Since ordered pairs of adjacent vertices in a simple graph are usually called {\em arcs}, we shall use the term {\em arc} as a synonym for a {\em dart}.

A \emph{graph morphism} $f \colon \Gamma \to \Gamma'$,  is a function $f\colon \V(\Gamma) \cup \D(\Gamma) \to \V(\Gamma') \cup \D(\Gamma')$  that maps $\V(\Gamma)$ to $\V(\Gamma')$ and  $\D(\Gamma)$ to $\D(\Gamma')$ such that  $f\,\beg_{\Gamma'} = \beg_\Gamma\,f$  and $f\,\inv_{\Gamma'} = \inv_\Gamma\,f$. A bijective morphism from $\Gamma$ to $\Gamma$ is called an {\em automorphism};  the set of all automorphisms is denoted by $\Aut(\Gamma)$.  The image of $x\in \V(\Gamma) \cup \D(\Gamma)$ under $\alpha\in \Aut(\Gamma)$ is denoted $x^\alpha$,  and the product $\alpha\beta$ of two automorphisms $\alpha, \beta \in \Aut(\Gamma)$ is defined by  $x^{(\alpha\beta)} = (x^\alpha)^\beta$; under this product, the set $\Aut(\Gamma)$ becomes a group, called the \emph{automorphism group} of $\Gamma$. If $G\leq \Aut(\Gamma)$, then we say that $\Gamma$ is \emph{$G$-vertex-, $G$-edge-, or $G$-arc-transitive}, if $G$ is transitive on the sets of vertices, edges, or darts (arcs) of $\Gamma$, respectively.

\subsection{Quotients, covering projections and voltage graphs}
\label{subsec:covproj}

For a group $G$ acting on a set $\Omega$, let $\Omega/G$ denote the set of orbits of $G$ on $\Omega$. Let 
$\Gamma$ be a connected graph and let $N \leq \Aut(\Gamma)$. In \cite{GP0} the quotient graph $\Gamma_N$
is defined as the simple graph whose vertex set is $\V(\Gamma)/N$, and two distinct orbits $v^N,u^N \in \V(\Gamma)/N$ are adjacent if and only if there exist $v' \in v^N$ and $u' \in u^N$ that are adjacent in $\Gamma$.
We refer to this  quotient as the {\em simple quotient of $\Gamma$ by $N$}. 
However, in this paper  it is convenient to use the following definition. The {\em quotient graph} $\Gamma/N$  {\em of $\Gamma$ with respect to $N$} is the quadruple 
$$
\Gamma/N = (\V(\Gamma)/N, \D(\Gamma)/N, \beg_N,\inv_N)
$$
where $\beg_N(x^N)=(\beg_\Gamma(x))^N$ and $\inv_N(x^N)=(\inv_\Gamma(x))^N$ for every $x\in \D(\Gamma)$. The function  $q_N \colon \V(\Gamma) \cup \D(\Gamma) \to \V(\Gamma)/N \cup \D(\Gamma)/N$ mapping $x$ to $x^N$ is then a surjective graph morphism, called the {\em quotient projection with respect to $N$}.

A surjective graph morphism $\p\colon \tilde\Gamma \to \Gamma$ of connected graphs is called a \emph{covering projection} if it is locally bijective, that is, if the neighbourhood $\tilde\Gamma(\tilde u)$ of each vertex $\tilde u\in \V(\tilde\Gamma)$ is mapped bijectively onto the neighbourhood $\Gamma(\tilde u f)$.  The graph $\tilde \Gamma$ is then called a \emph{covering graph} of the \emph{base graph} $\Gamma$, and the preimage $x \p^{-1}\subseteq \V(\tilde\Gamma)$ is called a {\em vertex fibre} if {\em $x\in\V(\Gamma)$}, and is called a {\em dart fibre} if $x\in \D(\Gamma)$. Two covering projections $\p\colon \tilde \Gamma\to \Gamma$ and $\p'\colon\tilde\Gamma'\to \Gamma$ are {\em isomorphic} if $\p\, \alpha =\tilde \alpha\,\p'$ for some automorphism $\alpha\in\Aut(\Gamma)$ and isomorphism $\tilde\alpha\colon \tilde\Gamma\to\tilde\Gamma'$. If $\alpha$ can be chosen to be the identity mapping $\id_\Gamma$, then the projections $\p$ and $\p'$ are {\em equivalent}.
The subgroup $K\leq\Aut(\tilde\Gamma)$ of all automorphisms fixing each vertex fibre and each dart fibre setwise is called the \emph{group of covering transformations} of $\p$. Note that $K$ is  semiregular in its action on the vertex set. 
If $K$ is moreover transitive (and hence regular) on each fibre, then $\p$ is a \emph{regular} (or more precisely, a {\em $K$-regular}) covering projection. 
Further, the quotient projection $q_K \colon \tilde{\Gamma} \to \tilde{\Gamma}/K$ is then a covering projection equivalent to $\p$. The following lemma (which fails to be true for the simple quotient $\Gamma_N$) holds. 

\begin{lemma}
\label{lem:semiregular}
Let $N$ be a group of automorphisms of a connected graph $\Gamma$ and let $q_N \colon \Gamma \to \Gamma/N$ be the corresponding quotient projection. Then $q_N$ is a covering projection if and only if $N$ acts semiregularly on $\V(\Gamma)$.
\end{lemma}

A convenient way of describing an $N$-regular covering projection is by using {\em voltage assignments}. For  an abstract group $N$,  a function $\zeta\colon \D(\Gamma)\to N$ with the property that $\zeta(x^{-1})=\zeta(x)^{-1}$ is called a \emph{voltage function}. 
Suppose henceforth that $N\cong \ZZ_p^n$ for some prime $p$ and positive integer $n$. Then the voltage assignment $\zeta$ extends naturally to 
closed  walks, which can be viewed as the elements of the first homology group  $H_1(\Gamma,\ZZ_p)$. In this sense, $\zeta$ induces a linear mapping $\zeta^* \colon H_1(\Gamma,\ZZ_p) \to N$.
 Conversely, given a linear mapping $\bar{\zeta} \colon H_1(\Gamma,\ZZ_p) \to N$ and a spanning tree $T$,
 there exists a unique voltage assignment $\zeta$ taking trivial value on each dart of $T$ and
 such that $\bar{\zeta} = \zeta^*$. 

Given a voltage function $\zeta$ on a graph $\Gamma$ such that $\zeta^*$  is surjective,  let  $\Cov(N;\zeta)$  be the \emph{derived graph} with vertex set $\V(\Gamma)\times N$, dart set $\D(\Gamma)\times N$, and  the functions  $\beg$ and $\inv$  defined by  $\beg(x,g)=(\beg(x),g)$ and $(x,g)^{-1}=(x^{-1}, g\zeta(x))$. Then the derived graph is connected  and 
 the mapping $\wp_\zeta \colon \Cov(N;\zeta) \to \Gamma$ defined by  $\wp_\zeta(x,g)=x$, for $x \in V(\Gamma) \cup D(\Gamma)$, is a regular $N$-covering projection. 
It is well known that every regular $N$-covering projection $\p\colon\tilde\Gamma\to \Gamma$ is  equivalent to some \emph{voltage projection} $\p_\zeta\colon \Cov(N;\zeta) \to \Gamma$ as defined above.

If $\bar{\zeta} \colon H_1(\Gamma,\ZZ_p) \to N$ is a surjective linear mapping and $T_1$ and $T_2$ are two
spanning trees, then the corresponding voltage assignments $\bar{\zeta}_1$ and $\bar{\zeta}_2$ give rise
to equivalent covering projections $\wp_{\zeta_1}$ and $\wp_{\zeta_2}$. In this sense, we may think of
voltage assignments as being defined on $H_1(\Gamma,\ZZ_p)$ (see \cite{MMP} for details).

\subsection{Lifting and projecting automorphisms} 
\label{subsec:elemab}

We say that the automorphism $g \in\Aut(\Gamma)$ \emph{lifts along} a covering projection $\wp\colon \tilde\Gamma\to \Gamma$ if 
there exists an automorphism $\tilde{g}\in\Aut(\tilde\Gamma)$ such that $ \tilde{g}\, \wp = \wp\,g$. The subgroup $G\leq\Aut(\Gamma)$ lifts, if every $g\in G$ lifts, in which case we say that $\wp$ is $G$-\emph{admissible}; the collection of all lifts forms a group, \emph{the lifted group} $\tilde{G}$. 

We briefly recollect  certain facts about lifting and projecting groups along a connected regular $K$-covering projection $\tilde\Gamma\to \Gamma$  to be used later on.
For more information we refer the reader to \cite{Djok, MNS00, MP16}.
First, a group $G$ and its lifted group $\tilde G$ share  certain symmetry properties  such as vertex-, edge- or arc-transitivity. Second, the largest group that projects is the normalizer $N_{\Aut(\tilde\Gamma)}(K)$ of the group of covering transformations within the full automorphism group of $\tilde\Gamma$, and $N_{\Aut(\tilde\Gamma)}(K)$ is  the lift of the largest group that lifts. Third, a vertex stabilizer $\tilde{G}_{\tilde{v}}$ projects isomorphically onto the vertex stabilizer $G_v$, $v = \p(\tilde{v})$, while the lift of a stabilizer $G_v$ is isomorphic to a semidirect product $G_v \ltimes K$.   

Characterization of $G$-admissible covering projections in terms of voltage assignments is an important classification method. In our case of elementary abelian covers  we shall apply results of \cite[Theorem 6.2 and Corollary 6.3]{MMP}, which we summarize by the following steps:
\begin{enumerate}
\item{\it Matrix representation.}
Consider the first homology group $\HG_1(\Gamma,\ZZ_p)$ as an $n$-dimensional vector space over some prime field $\ZZ_p$, where 
$n = \beta(\Gamma)$  is the Betti number of the graph. Then every $g \in\Aut(\Gamma)$ induces an invertible linear transformation on $\HG_1(\Gamma,\ZZ_p)$, represented in a chosen basis ${\cal B}$ by a matrix $g^\#\in\ZZ_p^{n \times n}$. This defines a homomorphism $\#\colon \Aut(\Gamma)\to \GL(n,\ZZ_p)$;  we denote the image of $G \leq \Aut(\Gamma)$  by $G^\#=\#(G)$ .
\item{\it Invariant subspaces.}
 For all $g\in G$ (or some generating set $\text{gen}(G)$), consider the transposed matrices $(g^\#)^t$ and find their common invariant subspaces. There  is a bijective correspondence between $(G^\#)^t$-invariant subspaces and equivalence classes of covering projections. (The isomorphism classes of projections are the orbits of the action of $\#(\Aut \Gamma)^t$ on the set of $(G^\#)^t$-invariant subspaces). Additionally,  minimal invariant subspaces correspond to {\em minimal covers} in the sense that these cannot be further decomposed into a series of $G$-admissible covering projections. (Note that a minimal invariant subspace is by definition nontrivial; the trivial subspace  corresponds to the trivial covering and is excluded from our considerations.)
\item{\it Voltage assignments.} 
For each $(G^\#)^t$-invariant subspace $V$ of $\ZZ_p^{1\times n}$, choose some basis ${\cal B}(V)$ and let $M_{{\cal B}(V)}\in\ZZ_p^{d\times n}$ be the matrix whose rows are the corresponding base vectors. Then the columns of the matrix $M_{{\cal B}(V)}$ define the values of the voltage assignment $\zeta$ on the basis ${\cal B}$ of $\HG_1(\Gamma,\ZZ_p)$. Now the values of $\zeta$ on the dart set $D(\Gamma)$ are easily computed with some degree of freedom (tipically, the values on ${\cal B}$ are interpreted as the values on the cotree darts for some chosen  tree $T\leq \Gamma$ and $0$ elsewhere). 
\end{enumerate}
Applying the above  procedure, all pairwise nonequivalent (even  pairwise nonisomorphic) $G$-admissible covering projections with values in some elementary-abelian group $\ZZ_p^d$ are obtained.

\subsection{Cyclic and related codes}
\label{subsec:codes}

Recall that a cyclic code over $\ZZ_p$ of length $n$ is a subspace of $\ZZ_p^n$, invariant under the transformation which cyclically shifts coordinates one step to the right. We will need a slight generalisation of this notion.

Let $\epsilon \in \{0,1\}$. Then an {\em $\epsilon$-cyclic code of length $n$} is a subspace $E$ of the vector space $\ZZ_p^n$,  invariant under the linear transformation $R_\epsilon\in \GL(n,\ZZ_p)$ that maps the standard basis $\{e_0, \ldots, e_{n-1}\}$ according to the rule 
$e_i \mapsto e_{i+1}$ for $i\in\{0, \ldots, n-2\}$ and $e_{n-1} \mapsto (-1)^\epsilon e_0$. (Note that a $0$-cyclic code is then just the usual cyclic code, and $1$-cyclic code is sometimes called a {\em negacyclic code}.) Recall that
$\Delta_{n,\epsilon}(x) = x^n - (-1)^\epsilon \in \ZZ_p[x].$
As in the theory of cyclic codes (see \cite{LN}, for example), it is convenient to identify elements of the vector space $\ZZ_p^n$ with the elements of the quotient polynomial ring $\ZZ_p[x]/(\Delta_{n,\epsilon}(x))$ via the identification 
\begin{equation}
\label{eg:identification}
(\alpha_0, \ldots, \alpha_{n-1})  \mapsto \alpha_0 + \alpha_1x + \ldots + \alpha_{n-1}x^{n-1}\> \in \ZZ_p[x]/(\Delta_{n,\epsilon}(x)),
\end{equation}
where an element $f(x) + (\Delta_{n,\epsilon}(x))$ of the quotient ring is represented simply by its representative $f(x) \in \ZZ_p[x]$. This is a standard convention that  we often use throughout the paper.  Under the above  identification, the action of $R_\epsilon$ on $\ZZ_p^n$ corresponds to  the multiplication by the polynomial $x \in \ZZ_p[x]/(\Delta_{n,\epsilon}(x))$. Consequently, the  $\epsilon$-cyclic codes are in bijective correspondence with the ideals in  $\ZZ_p[x]/(\Delta_{n,\epsilon}(x))$.
 
 For an ideal $\cI$ of $\ZZ_p[x]/(\Delta_{n,\epsilon}(x))$, let $g(x)$ be the nonzero monic polynomial  of smallest degree   such that $g(x) + (\Delta_{n,\epsilon}(x)) \in \cI$. Then $g(x)$ is a divisor of $\Delta_{n,\epsilon}(x)$,  $\cI$ is generated by $g(x) + (\Delta_{n,\epsilon}(x))$, and $g(x)$ is  called the {\em generating polynomial} of the $\epsilon$-cyclic code $E$ that corresponds to $\cI$.  If
$$
g (x) =  \alpha_0 + \alpha_1x + \ldots + \alpha_mx^m
$$
is the generating polynomial of an $\epsilon$-cyclic code $E$, then  $\alpha_0 \not = 0$, the dimension of $E$ is $\dim E =n-m$,  and $E$ is generated by the rows of the matrix $M_{g(x)}$ as given in Definition~\ref{def:matrix}.

\subsection{Some remarks on polynomials}
\label{subsec:polynomials}

\bigskip
In this subsection we state certain lemmas about polynomials in $\ZZ_p[x]$ that will be used in the proof of Theorem~\ref{thm:maxG-kernel}. We omit some of the obvious proofs. 
Recall that $\expn(f(x))$ denotes the set of all positive integers $k \in \NN$ such that 
$f(x)$ is a polynomial in $x^k$, and  if $\deg(f(x)) = n\geq 1$, then  $d = \Expn(f(x)) = \max\{k\ |\ k \in \expn(f(x))\} \leq n$.

\begin{lemma}
\label{lem:poly-in-d}
Let $f(x) = g(x)h(x)$, where  $f(x), g(x)$ and $h(x)$ are nonconstant polynomials in $\ZZ_p[x]$. If  $f(x)$ and $g(x)$ are polynomials in $x^k$, then $h(x)$ is a polynomial in $x^k$. 
$\hfill \square$
\end{lemma}

\comment{
\begin{proof}
Denote  by  $\{\alpha_i \ |\ i = 0, \ldots, \deg\,f(x)\}$,  $\{\beta_i \ |\ i = 0, \ldots, \deg\,g(x)\}$, and $\{\gamma_i \ |\ i = 0, \ldots, \deg\,h(x)\}$ the coefficients of $f(x)$, $g(x)$, and $h(x)$respectively. Let $i$ be the largest index not divisible by $k$ such that 
$\gamma_i \neq 0$. Then 
$$
\sum_{j= 0}^{\deg\,g(x)} \beta_{\deg\,g(x) - j} \gamma_{i + j} = \alpha_{\deg\,g(x) + i}.
$$
 Since $k$ is not a divisor of $i$  we have $\alpha_{\deg\,g(x) + i} = 0$. Now, for $j \neq 0$, if $k$ does not divide $j$ then $\beta_{\deg\,g(x) - j} = 0$, while  if $d$ divides $j$ then $\gamma_{i+j} = 0$ (either by maximality of $i$ or else because $i + j > \deg\,h(x)$). It follows that $\beta_{\deg\,g(x)} \gamma_i = 0$, and so $\gamma_i = 0$.
This contradiction shows that $\gamma_i = 0$  for all $i$ that are not divisible by $k$. Hence $h(x)$ is a polynomial in $x^k$. 
\end{proof}
} 

\begin{lemma}
\label{lem:d-divides-n}
Let $g(x), h(x) \in\ZZ_p[x]$ be nonconstant polynomials such that $g(x) h(x) =\Delta_{n,\epsilon}(x)$. If  $g(x)$ is a polynomial in $x^k$, then $k$ is a divisor of $n$, and $h(x)$ is also a polynomial in $x^k$. In particular, let $d = \Expn(g(x))$. Then  $\Expn(h(x)) = d$, and $g_d(x) h_d(x) = \Delta_{n/d, \epsilon}(x)$.
\end{lemma}

\begin{proof}
Let  $g(x) = \alpha_0 + \alpha_k x^k + \ldots + \alpha_{rk} x^{rk}$ and  $h(x) = \beta_0+\beta_1x+\ldots +\beta_tx^t$, $\beta_t \neq 0$. Note that $\alpha_0, \beta_0 \neq 0$ since $\Delta_{n,\epsilon}$ is not divisible by $x$.  We have 
$$
(\alpha_0+\alpha_kx^k+\ldots+\alpha_{rk}x^{rk})\cdot (\beta_0+\beta_1x+\ldots +\beta_tx^t) =x^n - (-1)^{\epsilon}.
$$
Let $i\in \{1,\ldots,t\}$ be the smallest index not a multiple of $k$ such that  $\beta_i\ne 0$. Then, $\alpha_0\beta_ix^i$ is the only term with $x^i$ on the right-hand side. Since $\alpha_0 \neq 0$ we must have   $\beta_i=0$,  a contradiction. So $t = sk$ and $h(x) = \beta_0 + \beta_k\, x^k + \ldots + \beta_{sk}\, x^{sk}$. Consequently, $rk + sk = n$. Hence $k$ is a divisor of $n$ and $k \in \expn(h(x))$.

Now let $d = \Expn(g(x))$ and $d' = \Expn(h(x))$. By the first part we have $d \in \expn(h(x))$ and  $d' \in \expn(g(x))$. By the maximality of $d'$ and $d$ we have 
 $d \leq d'$ and $d' \leq d$, so $d = d'$, as required. The equality 
$g_d(x) h_d(x) = \Delta_{n/d, \epsilon}(x)$ is then immediate.
\end{proof}

\begin{remark}
\label{rem:g-is-1}
{\rm 
The last statement in Lemma~\ref{lem:d-divides-n} also holds in the degenerate case when $g(x) =1$ and $h(x) = \Delta_{n,\epsilon}$.  Then $d = n$, $g_d(x) = 1$, and $h_d(x) = \Delta_{1,\epsilon}$.
} 
\end{remark}

\bigskip
For a polynomial $f(x) = \alpha_0 + \alpha_1 x + \ldots + \alpha_mx^m \in \ZZ_p[x]$ we denote by $\ind\,f(x) = \{ i \ |\ a_i \neq 0\}$ the set of indices of all nonzero coefficients of $f(x)$.   

\begin{lemma}
\label{lem:Expn-gcd}
A nonconstant polynomial $f(x) \in \ZZ_p[x]$ can be written as a polynomial in $x^k$ if and only if the ideal $(k) \subseteq \ZZ$ contains $\ind\,f(x)$. Moreover, the largest such exponent $k$ is equal to the greatest common divisor of integers in $\ind\,f(x)$, that is,  $d = \Expn(f(x)) = \gcd(\ind\,f(x))$,  and $k \in \expn(f(x))$ if and only if $k$ is a divisor of $d$.
$\hfill\square$
\end{lemma}

\comment{ 
\begin{proof}
Let $f(x) = \alpha_0 + \alpha_1x + \ldots + \alpha_mx^m  \in \ZZ_p[x]$ be  a polynomial in $x^k$. If $\alpha_i \neq 0$ then $i$ is divisible by $k$, and so $\ind \,f(x) \subset (k)$. Conversely, let $\ind\,f(x) \subset (k)$. Then the index of each nonzero coefficient is divisible by $k$, and so $f(x)$ is a polynomial in $x^k$.
$\hfill\square$

The smallest ideal containing $\ind\,f(x)$ is $(d)$, where $d = \gcd(\ind\,f(x))$, and this smallest ideal  is the intersection of all ideals $(k)$ that contain $\ind\,f(x)$.  Thus, for each $k \in \expn(f(x))$ we have $(d) \subseteq (k)$, and so $k$ is a divisor of $d$. Conversely, if $k$ divides $d$, then $(d) \subseteq (k)$. Hence $(k)$ contains $\ind(f(x))$ and $k \in \expn(f(x))$, as required.
\end{proof}
} 

\begin{lemma}
\label{lem: Q-ofx}
Let $f(x) \in \ZZ_p[x]$ be a nonconstant polynomial. For  $d = \Expn(f(x))$ and $k \in \expn(f(x))$, let $Q(x) = f_d(x^{d/k})$ and let $d'= \Expn(Q(x))$. Then  $d' = d/k$  and $Q_{d'}(x) = f_d(x)$.
\end{lemma}

\begin{proof}
By Lemma~\ref{lem:Expn-gcd} we have that $k$ is a divisor of $d$, so the polynomial 
$Q(x)$ exists.
From $f_d(x^{d/k}) = Q(x) = Q_{d'}(x^{d'})$ it follows, by Lemma~\ref{lem:Expn-gcd}, that $d/k$ divides $d'$. Consequently, there is a natural number $N$ such that 
$Nd = kd'$. From  $f(x) = f_d(x^d) = Q(x^k)= Q_{d'}(x^{kd'}) = Q_{d'}(x^{Nd})$ it then follows, also by Lemma~\ref{lem:Expn-gcd}, that $Nd$ divides $d$. Hence $N = 1$, and so $d' = d/k$. Finally, $f_d(x^{d'}) = Q(x) = Q_{d'}(x^{d'})$ implies
$Q_{d'}(x) = f_d(x)$, and the proof is complete. 
\end{proof}

\begin{lemma}
\label{lem:refl-unique}
A  nonzero polynomial $f(x)  \in \ZZ_p[x]$  is type-1 and type-2 reflexible if and only if $2 \in \expn(f(x))$, and in this case the two types of reflexibility coincide.  In particular, the reflexibility type of a reflexible polynomial  is uniquely determined. 
\end{lemma}

\begin{proof}
 Let 
$f(x) = \alpha_0 + \alpha_1x +\ldots + \alpha_m x^m \in \ZZ_p[x]$.  
Suppose that $\lambda \alpha_{m-i}= \alpha_i$ and 
$\mu \alpha_{m-i}= (-1)^i \alpha_i$ for all $i$. For $i = 0$ we have 
$\lambda \alpha_m =\alpha_0 = \mu \alpha_m$. Now  $\alpha_m \neq 0$ implies
$\lambda =\mu$, and $\alpha_i = \lambda \alpha_{m-i} = (-1)^i \alpha_i$ forces that $(1 - (-1)^i) \alpha_i = 0$ for all indices. For $i$ odd we have $2\alpha_i =0$ and hence $\alpha_i = 0$ since $2 \in \ZZ_p$ is invertible for $p$ odd. Hence $2 \in \expn(f(x))$. Conversely, if $2 \in \expn(f(x))$ then $a_i = 0$ for all odd indices $i$, and then there is no difference between both types of reflexibility. 
The last statement in the lemma holds either because there is no difference between the two types of redlexibility or else because $2 \not\in \expn(f(x))$, and $f(x)$ is either type-1 or type-2 reflexible.
\end{proof}

\begin{lemma}
\label{lem:f-refl-fd-refl}
Let $f(x) \in \ZZ_p[x]$ be a polynomial of degree $m \geq 1$. If $\Expn(f(x))$ is odd, then $f(x)$ is reflexible if and only if $f_d(x)$ is reflexible; the reflexibility type is preserved. If $\Expn(f(x))$ is even, then $f(x)$ is reflexible if and only if $f_d(x)$ is reflexible of type $1$. In particular, a reflexible polynomial is weakly reflexible; the converse holds unless $f_d(x)$ is type-2 reflexible with $d$ even. 
\end{lemma}

\begin{proof}
Let $f(x) = \alpha_0 + \alpha_1 x + \ldots + \alpha_mx^m = \alpha_0 + \alpha_dx^d + \ldots + \alpha_{td} x^{td}$. Letting $\alpha_i' = \alpha_{id}$ we have 
$f_d(x) = \alpha_0' + \alpha_1' x + \ldots + \alpha_t'x^t$.

First suppose that $f(x)$ is type-1 reflexible with $\lambda \alpha_{m-i} = \alpha_i$. For $i = jd$ we have  $\lambda \alpha_{td - jd} = \alpha_{jd}$. Hence $\lambda \alpha_{t-j}' = \alpha_j'$, and so $f_d(x)$ is type-1 reflexible. The implications hold in the other direction as well.

Suppose now that $f(x)$ is type-2 reflexible with $\lambda \alpha_{m-i} = (-1)^i\alpha_i$. Here we must distinguish two cases. First, let $d$ be odd. For $i = jd$ we have $\lambda \alpha_{td-jd} = (-1)^{jd}\alpha_{jd} = (-1)^{j}\alpha_{jd}$. Hence
$\lambda \alpha_{t-j}' = (-1)^{j}\alpha_j'$, and $f_d(x)$ is type-2 reflexible. The implications hold in the other direction as well.
Second, let $d$ be even. Then $f(x)$ is also type-1 reflexible by Lemma~\ref{lem:refl-unique}, and so $f_d(x)$ is type-1 reflexible. Clearly, if $d$ is even and $f_d(x)$ is type-2 reflexible, then $f(x)$ cannot be reflexible. For if it were, then $f(x)$ would be type-1 reflexible and hence $f_d(x)$ also  type-1 reflexible. But in view of Lemma~\ref{lem:refl-unique}  this is a contradiction since $\Expn(f_d(x)) = 1$.
\end{proof}


\section{Proofs of Theorems}
\label{sec:Proof-4-5-3}
\subsection{Vertex- and edge-transitive subgroups of $\Aut \DC$}
\label{subsec:AutDC}

In this section we describe the conjugacy classes of subgroups of automorphisms with transitive action on both the vertex and the edge set, that is, the conjugacy classes of half-arc-transitive and arc-transitive subgroups of $\Aut(\DC)$. 

Recall that the vertex set of $\DC$ is indexed by $\ZZ_n$, and that the two arcs (darts) from a vertex $i$ to the vertex $i+1$ are denoted by $a_i$ and $b_i$ while  $a_i^{-1}$ and $b_i^{-1}$ are the corresponding inverse arcs. 

Automorphisms of $\DC$ are best described as  permutations of arcs (where the action on inverse arcs  is suppressed for brevity). Consider 
the following ones:   the \emph{elementary transposition}  exchanging the arcs $a_j$ and $b_j$,  the \emph{reflection} $\sigma$ across the arcs connecting the vertices $0$ and $1$, and the \emph{one step rotation} $\rho$,
$$
\tau_j=(a_j b_j) = \tau_0^{\rho^j} = \rho^{-j}  \tau_0\, \rho^j ,
$$
$$
{a_i}^\sigma = (a_{-i})^{-1},\> {b_i}^\sigma=(b_{-i})^{-1},
$$
$$
\rho=(a_0 a_1 \ldots a_{n-1})(b_0 b_1 \ldots b_{n-1}).
$$

Observe that $\rho$, $\tau_0$, and $\sigma$ generate the full automorphism group $\Aut(\DC)$, and that  
$$
K = \gen{\tau_0, \tau_1, \ldots, \tau_{n-1}} \cong  \ZZ_2^n
$$
is  normal in  $\Aut(\DC)$; in fact, $K$ is the kernel of the action of $\Aut(\DC)$ on the vertex-set $V(\DC)$. Denoting $\tau_J = \prod_{j \in J} \tau_j $ we have 
$$
K = \{\tau_J  : J \subseteq \ZZ_n\}  \quad\text{with}\quad \tau_J\tau_L = \tau_{J\oplus L},
$$
where $J\oplus L$ is the symmetric difference of $J$ and $L$ in $\ZZ_n$. The group $K$ is occasionally  considered as a $\ZZ_2$-vector space. Let   $\chi\colon K \to \ZZ_2^n$ be the {\em characteristic function} assigning to each $\tau_J$ its {\em characteristic vector} 
$\chi_J = \chi(\tau_J) = (\chi_0^J,\chi_1^J, \ldots, \chi_{n-1}^J)$, where
\begin{equation}
\label{eq:charact}
\chi_i^J= \left\{\begin{array}{rl}
                         1 & i  \in J, \\
                         0 & i  \notin J.
                  \end{array}\right.
\end{equation}
Then  $\chi_i^{J \oplus L} = \chi_i^J + \chi_i^L$, and the characteristic function is clearly a group as well as a vector space  isomorphism. 

The group $\Aut(\DC)$ is isomorphic to a semidirect product of   $K$ by the dihedral group $\gen{\rho, \sigma}$, 
$$
\Aut(\DC)=\gen{\rho,\tau_0,\sigma} = \gen{\rho, \sigma}  \ltimes K  \cong \Dih(n) \ltimes  \ZZ_2^n.
$$
With the convention that  $J +k  = \{j+k : j \in J\}$, for $k\in \ZZ_n$ and $J\subseteq \ZZ_n$, the action of $\rho$ on $K$ is explicitly given by 
\begin{equation}
\tau_J^{\rho} = \tau_{J + 1},\quad   \chi_i^{J+1} = \chi_{i-1}^J. \label{eq:Jrho}
\end{equation}
In view of the isomorphism  $\chi\colon K \to \ZZ_2^n$ the rotation  $\rho$ acts on the vector space $\ZZ_2^n$ as the {\em cyclic shift } ${\cal C} \colon (\chi_0, \chi_1, \ldots, \chi_{n-1}) \mapsto (\chi_{n-1}, \chi_0, \ldots, \chi_{n-2})$. Next,
 with the convention that  $-J = \{-j : j \in J\}$,   the action of $\sigma$ on $K$ is
\begin{equation}
\tau_J^{\sigma} = \tau_{-J}, \quad  \chi_i^{-J} = \chi_{-i}^J. \label{eq:Jsigma}
\end{equation}
 Finally,  the group $\Aut_0(\DC) = \langle \rho, \tau_0\rangle $ is clearly an index-$2$ subgroup of  $\Aut(\DC)$, and
$$
\Aut_0(\DC) = \gen{\rho}   \ltimes  K    \cong  \ZZ_n \ltimes \ZZ_2^n.
$$

\begin{lemma}
\label{lem:tau-C}
Let $\tau_C \in K$. Then there exists $X \subseteq \ZZ_n$ such that $\tau_C = \tau_X\tau_{X+1}\tau_0^{\epsilon}$ where $\epsilon = 0$ if $|C|$ is even and $\epsilon =1$ if $|C|$ is odd.
\end{lemma}

\begin{proof}
Consider $K$ as a $\ZZ_2$-vector space. Then the function  $\ad_\rho\colon K\to K$,  $\tau_X \mapsto \tau_X\tau_{X+1}$ is a linear mapping (it corresponds to $\id + {\cal C}$, where ${\cal C}$ is the cyclic shift). Its  kernel is $1$-dimensional, generated by $\tau_{\ZZ_n}$. Thus, the image of $\ad_\rho$ has dimension $n-1$ and is generated by $\tau_0\tau_1, \ldots,  \tau_{n-2}\tau_{n-1}$. It follows that the image of $\ad_\rho$ is precisely the set of all those $\tau_Y$ for which $|Y|$ is even. Consequently,  there exists some $X\subseteq \ZZ_n$ such that $\tau_C = \tau_X\tau_{X+1}\tau_0^{\epsilon}$ where $\epsilon = 0$ if $|C|$ is even and $\epsilon =1$ if $|C|$ is odd, as required. 
\end{proof}

Our characterization of vertex- and edge-transitive subgroups of $\Aut(\DC)$ is now given as follows.

\begin{lemma}
\label{lem:AT}
Let $G$ be a subgroup of $\Aut(\DC)$. Then $G$ is vertex- and edge-transitive if and only if one of the following holds:
\begin{itemize}
\item[(i) ]
$G$ is not arc-transitive and $G$ is conjugate in $\Aut(\DC)$ either to the group 
\begin{equation}\label{eq:nonrotary}\gen{\rho, \sigma\tau_{\ZZ_n}},
\end{equation}
or to the group
\begin{equation}
\label{eq:AT1}
\langle B, \rho\tau_0^\epsilon \rangle,
\end{equation}
where $B \leq K$ is nontrivial and normal in $G$, and $\epsilon \in \{0,1\}$ is such that $\tau_{\ZZ_n}^{\epsilon} \in B$;

\item[(ii)]
$G$ is arc-transitive and $G$ is conjugate in $\Aut(\DC)$ to the group
\begin{equation}
\label{eq:AT2}
\langle B, \rho\tau_0^{\epsilon}, \sigma\tau_0^\epsilon\tau_J\rangle
\end{equation}
where $B \leq K$ is nontrivial and normal in  $G$, and $\epsilon\in\{0,1\}$ is such that $\tau_{\ZZ_n}^\epsilon \in B$, while $J\subseteq \ZZ_n$ is such that $\tau_J\tau_{-J} \in B$  and $\tau_J\tau_{J+1}\in B$. 
\end{itemize}
Moreover, $G\cap K=1$ in case $(\ref{eq:nonrotary})$ while $G\cap K=B$ in cases  $(\ref{eq:AT1})$ and  $(\ref{eq:AT2})$. Furthermore,  in case  $(\ref{eq:AT1})$ the group $B$ is normal in  $\Aut_0(\DC)$,  and normal in $\Aut(\DC)$  in case  $(\ref{eq:AT2})$.
\end{lemma}

\begin{proof}
Let us first prove that the groups as in (\ref{eq:nonrotary}), (\ref{eq:AT1}), and (\ref{eq:AT2}) have the required properties. The group as in (\ref{eq:nonrotary}) is obviously vertex and edge but not arc transitive, and $G \cap K = 1$.

\bigskip
Suppose that $G$ is as in $(\ref{eq:AT1})$. First, since $B$ is normal in $G$ by assumption, $(\rho \tau_0^\epsilon)^{-1} B  (\rho \tau_0^\epsilon) = B$
implies $\rho^{-1} B \rho = B$. Hence $B$ is normal in $\Aut_0(\DC)$.  Now, as $B$ is nontrivial  there exists  $\tau_L \in B$ switching at least one pair of arcs $a_i$ and $b_i$. But since    $\tau_{L + k} = \tau_L^{\rho^k} \in B$ for all $k \in \ZZ_n$,  each pair of parallel  arcs is switched under the action of $B$. Moreover, $\rho\tau_0^{\epsilon}$   makes the action of $G$ vertex- and edge-transitive. Clearly, $G$ is not arc.transitive. 
Let us prove that  $G \cap K = B$. Since $B$ is normalized by $\rho\tau_0^{\epsilon}$, any element in $G$ is of the form $g = \tau_L^i (\rho\tau_0^{\epsilon})^j$, where $\tau_L \in B$.  Suppose that $g \in G \cap K$. Then $(\rho\tau_0^{\epsilon})^j \in K$. Since  $(\rho\tau_0^{\epsilon})^j = \rho^j \tau_{j-1}^{\epsilon} \ldots \tau_{1}^{\epsilon}\tau_{0}^{\epsilon}$  it follows that  $(\rho\tau_0^{\epsilon})^j \in K$ is either trivial, or else equal to $(\rho\tau_0^{\epsilon})^j = \tau_{\ZZ_n} \in B$.
So $g \in B$ in all cases, and  $G \cap K = B$.

\bigskip
Suppose that $G$ is as in (\ref{eq:AT2}). The action of $\la B, \rho\tau_0^{\epsilon}\ra$ is vertex- and edge-transitive, and  $\sigma \tau_0^{\epsilon} \tau_J$ switches at least one arc $a_i$ with $a_i^{-1}$ or $b_i^{-1}$; hence $G$ is arc transitive. As before, $B$ is normalized by $\rho$ while  $(\sigma \tau_0^{\epsilon} \tau_J)^{-1} B (\sigma \tau_0^{\epsilon} \tau_J) = B$ implies 
$\sigma^{-1} B \sigma = B$, and so $B$ is normal in $\Aut(\DC)$.
To prove that $G \cap K = B$, observe that 
\begin{eqnarray}
\tau_{J + k}\tau_{J+l}, \ \tau_{J + k}\tau_{-J+l} & \in B, & \text{for all}\quad  k,l \in \ZZ_n  \label{eq:J-kl}.
\end{eqnarray}
This follows since $\tau_J\tau_{-J}$ and  $\tau_J\tau_{J+1}$ are in $B$ and since $B$ is normalized by  $\rho$ and $\sigma$. Next, we have
$$
(\sigma \tau_0^{\epsilon}\tau_J) (\rho \tau_0^{\epsilon}) (\sigma \tau_0^{\epsilon}\tau_J)^{-1} =  (\tau_{-J}\tau_{-J-1}) (\rho \tau_0^{\epsilon})^{-1}.
$$
In view of  (\ref{eq:J-kl})  it follows that  $\sigma \tau_0^{\epsilon}\tau_J$ normalizes $\la B, \rho\tau_0^{\epsilon}\ra$. Therefore, any element of $g \in G$ is of the form
$$ 
g = (\rho\tau_0^{\epsilon})^i \tau_L^j (\sigma \tau_0^{\epsilon}\tau_J)^k,
$$
where $\tau_L \in B$. Suppose now that  $g \in G \cap  K$. Then $k$ must be even for otherwise $g$ does not belong to $\Aut_0(\DC)$. Since $\sigma \tau_0^{\epsilon}\tau_J$ is of order $2$ or $4$, either $(\sigma \tau_0^{\epsilon}\tau_J)^k$ is trivial or else equal to $(\sigma \tau_0^{\epsilon}\tau_J)^2 = \tau_J\tau_{-J} \in B$.  Consequently, $(\rho\tau_0^{\epsilon})^i \in K$. But then  $(\rho\tau_0^{\epsilon})^i$ is either trivial or else equal to $(\rho\tau_0)^n = \tau_{\ZZ_n} \in B$. So $g \in G \cap K$ implies $g \in B$, and $G \cap K = B$, as required.

\bigskip
For the converse we show that any vertex- and edge-transitive group of $\DC$ is conjugated to one of the groups (\ref{eq:nonrotary}), (\ref{eq:AT1}), or (\ref{eq:AT2}).
Let $G\leq \Aut(\DC)$ be vertex- and edge-transitive, and let $B=K\cap G$. Then $B$ is the kernel of the action of $G$ on $V(\DC)$ and hence normal in $G$. Since $G$ is vertex- and edge-transitive, it follows that $G$ contains an automorphism which acts on $V(\DC)$ as the rotation $(0,1,\ldots,n-1)$, and is thus of the form $\rho\tau_C$ for some $\tau_C \in K$.
Now, since $B$ is normal in $G$ we have $(\rho \tau_C)^{-1} B (\rho \tau_C) = B$. Hence $\rho^{-1}B \rho = \tau_C B \tau_C^{-1} = B$, and $B$ is normal in $\Aut_0(\DC)$. Furthermore, by Lemma~\ref{lem:tau-C} we have $\tau_C = \tau_X\tau_{X+1}\tau_0^{\epsilon}$ for some $\tau_X \in K$ and $\epsilon \in \{0,1\}$. Then
$$
(\rho\tau_C)^{\tau_X} = \rho \tau_X\tau_{X+1}\tau_C=\rho\tau_0^{\epsilon}.
$$
It follows that  $G$ is conjugated within $\Aut \DC$ to  a group containing  $ \langle B, \rho\tau_0^{\epsilon}\rangle$. Assuming without loss of generality  that $G$ itself contains $\gen{B, \rho\tau_0^\epsilon}$,  let us look at the $G$-orbit of the arc $a_0$. 

\bigskip
Suppose first that $G$ is not arc-transitive. Then the  $G$-orbit $\Orb(a_0)$ of the arc $a_0$  is either $\{a_i,b_i^{-1}: i\in \ZZ_n\}$ or $\{a_i,b_i: i\in \ZZ_n\}$. 

If $\Orb(a_0) = \{a_i,b_i^{-1}: i\in \ZZ_n\}$  we must have $B=G\cap K = 1$  and $\rho\tau_0\notin G$.  Hence $\epsilon=0$ and  $\rho\in G$. Also, the orbit defines a digraph consisting of two directed cycles with inverse directions. Its automorphism group is $\gen{\rho, \sigma\tau_{\ZZ_n}} \cong \Dih(n)$, and this must also be the group $G$. 

If $\Orb(a_0) =\{a_i,b_i: i\in \ZZ_n\}$ we must have $G\le \langle \rho, \tau_0\rangle$. Moreover,  $B$ is nontrivial. For if  $B=1$,  then $(\rho\tau_0^{\epsilon})^n=\tau_{\ZZ_n}^{\epsilon} \in B$ implies $\epsilon=0$.  Hence $ \langle B, \rho\tau_0^{\epsilon}\rangle \leq G$ contains $\rho$ and, since $G$ acts transitively on the edge set, it must also contain an element $\rho\tau_L$, $\tau_L \neq 1$, which implies $\tau_L \in B$, a contradiction. Consequently, $B$ is nontrivial.  To show that 
$G = \la B, \rho\tau_0^{\epsilon}\ra$, recall that every pair of parallel arcs is switched under the action of $B$ since $B$ is normalized by $\rho$. Thus, the quotient $\DC/B$ is a simple cycle, and every element of $g \in G\setminus B$ projects to a rotation. So $g (\rho\tau_0^{\epsilon})^{-1} \in B$, and hence $G = \la B, \rho\tau_0^{\epsilon}\ra$. Finally,  $\tau_{\ZZ_n}^{\epsilon} = (\rho\tau_0^{\epsilon})^n$ implies $\tau_{\ZZ_n}^{\epsilon} \in B$, as claimed. This proves part (i). 

\bigskip
Suppose now that $G$ is arc-transitive. Then $G\cap \langle \rho,\tau_0\rangle$ has index $2$ in $G$ and is vertex- and edge-transitive; its    orbit of $a_0$ is $\{a_i,b_i: i\in \ZZ_n\}$. By (i)  we may without loss of generality assume that $G\cap \langle \rho,\tau_0\rangle = \langle B, \rho\tau_0^\epsilon \rangle$, where $B=G\cap K$ is a nontrivial  normal subgroup of $\Aut_0(\DC)$, and $\tau_{\ZZ_n}^{\epsilon} \in B$. Now, since $G$ is arc-transitive, $G$  contains an element of the form $\sigma\tau_J$ for some $\tau_J \in K$, which can be further written  (by modifying $J$ accordingly) as 
$$%
G = \langle B, \rho\tau_0^{\epsilon}, \sigma\tau_0^\epsilon\tau_J\rangle.
$$%

Next, since $K$ is normal in $\Aut(\DC)$, the intersection $B = G \cap K$ is normal in $G$. Hence 
$(\sigma\tau_0^\epsilon\tau_J)^{-1} B\,(\sigma\tau_0^\epsilon\tau_J) = B$. Thus, $\sigma^{-1} B\, \sigma = \tau_0^\epsilon\tau_J\,B\,\tau_0^\epsilon\tau_J = B$. As $B$ is normalized by $\rho$, this shows that $B$ is normal in $\Aut(\DC)$.
Finally, observe  that $G/B$ acts arc-transitively on the quotient $\DC/B$, which is a simple cycle of length $n$. In particular, $G/B \cong \Dih(n)$. Since $\sigma\tau_0^\epsilon\tau_J$ and $\rho\tau_0^\epsilon\, \sigma\tau_0^\epsilon\tau_J$ project to reflections, $(\sigma\tau_0^\epsilon\tau_J)^2 = \tau_J\tau_{-J}$ and $(\rho\tau_0^\epsilon\, \sigma\tau_0^\epsilon\tau_J)^2 = \tau_{-J-1}\tau_J$ are  elements of $B$. Now $B=B^\sigma$ implies  $\tau_{J+1}\tau_{-J} = (\tau_{-J-1}\tau_J)^\sigma \in B$. But since $\tau_J\tau_{-J}\in B$, we have 
$\tau_J\tau_{J+1} \in B$. This shows part (ii), completing the proof.
\end{proof}

\bigskip
An integer $k$, $1 \leq k \leq n $, is called a {\em period} of  a subgroup $H\le K$ provided that $H$ is centralised by $\rho^k$; equivalently, $\tau_{L+k} = \tau_L$ for all $\tau_L \in H$. The set of all periods  forms a subgroup of $\ZZ_n$ (with understanding that $n =0$). We call the smallest positive period  the {\em exact period of $H$}, and $H$  is then {\em $k$-periodic} (an $n$-periodic group is in fact aperiodic since its group of periods is trivial).  Note that the  exact period is  a divisor of $n$ as it is the `canonical' generator of the group of periods.  For a divisor $k$ of $n$ we let 
$$
\tau_{[i,k]} = \tau_i\tau_{i+k}\tau_{i+2k} \ldots \tau_{i+{(n/k-1)}k}.
$$
Clearly, $\langle \tau_{[i,k]} \rangle$ has exact period $k$; the largest $k$--periodic subgroup of $K$ is 
$\langle   \tau_{[0,k]}, \tau_{[1,k]}, \ldots , \tau_{[k-1,k]} \rangle \cong \ZZ_2^k$ and contains all 
$k$-periodic subgrups  of $K$.

\begin{lemma}
\label{lem:special-J}
Let $G =\langle B, \rho\tau_0^{\epsilon}, \sigma\tau_0^\epsilon\tau_J\rangle$ be an arc transitive group of $\DC$ as in part (ii) of 
Lemma~\ref{lem:AT}. Suppose that  $B$ has exact period  $k \geq 1$.  If $B$ is equal to the largest subgroup of $K$ with exact period $k$, then 
$G =\langle B, \rho\tau_0^{\epsilon}, \sigma\tau_0^\epsilon\tau_L\rangle$, where
\begin{center}
either $\tau_L=\tau_{\ZZ_n}$,  or else  
$2k \mid  n$  and $\tau_L= \Pi_0^{k-1} \tau_{[i,2k]}$.
\end{center}
\end{lemma}

\begin{proof}
Let $\chi(\tau_J) = (\chi_0, \ldots, \chi_{n-1})$. 
Then the $i$-th component of $\chi(\tau_J\tau_{J+1})$ is equal to $\chi_i +\chi_{i-1}$. Since $\tau_J\tau_{J+1} \in B$ and $B$ has exact period $k$  we have $\chi_i +\chi_{i-1} =  \chi_{i+k} + \chi_{i-1+k}$ for all $i$. This immediately implies that either $\chi_{i+k} = \chi_{i}$  holds for all indices or else $\chi_{i+k} = 1 +\chi_{i}$   holds for all indices. 

In the first case,  the `pattern' in  $J$ between $0$ and $k-1$ repeats between $ik$ and $(i+1)k -1$ for any $i$. In other words,  $\tau_J = \Pi_{i \in I} \tau_{[i,k]}$, where $I = \{0,1,\ldots, k-1\} \cap J$. Hence $\tau_J$ belongs to $B$ since $B$ is the largest subgroup with exact period $k$.   

In the second case,  the `pattern' between `intervals' $[(i-1)k, ik -1]$ and $[ik, (i+1)k -1]$ in $J$ are complementary for all $i$, and $\chi_{i+2k} = \chi_i$. But then $2k$ must divide $n$.

 Let  $\tau_C = \Pi_{i \in I} \tau_{[i,k]}$, where $I = \{0,1,\ldots, k-1\} \setminus J$,
and set $\tau_L = \tau_J\tau_C = \tau_{J \oplus C}$. Properties of $\tau_J$ described above imply that 
\begin{center}
either $\tau_L=\tau_{\ZZ_n}$, or else  
$2k \mid  n$  and $\tau_L= \Pi_0^{k-1} \tau_{[i,2k]}$.
\end{center}
Now, since $B$ is the largest subgroup with exact period $k$ we have $\tau_C \in B$. Hence $\tau_L\tau_{L+1} = \tau_J\tau_{J+1}\tau_C\tau_{C+1} \in B$ since $\tau_J\tau_{J+1} \in B$ and $\tau_C, \tau_{C+1} \in B$. Similarly, $\tau_L\tau_{-L} = \tau_J\tau_{-J} \tau_C\tau_{-C} \in B$.  Moreover, $G =\langle B, \rho\tau_0^{\epsilon}, \sigma\tau_0^\epsilon\tau_J\rangle = \langle B, \rho\tau_0^{\epsilon}, \sigma\tau_0^\epsilon\tau_L\rangle$, as required.
\end{proof}

\comment{ 
Moreover, observe that $\tau_J \tau_{J+1} \in B$ automatically implies 
$\tau_J\tau_{-J} \in B$. Indeed. If $\chi(\tau_J) = (\chi_0,\chi_1, \ldots, \chi_{n-1})$ then $\chi(\tau_{-J}) = (\chi_0,\chi_{n-1}, \ldots, \chi_1)$, 
and the $i$-th component of $\chi(\tau_J\tau_{-J})$ is equal to  $\chi_i + \chi_{n-i}$. If $J = \ZZ_n$ then 
$\tau_J \tau_{-J} = \id \in B$. Otherwise we have $\chi_{i+k} = 1 +\chi_i$ for all $i$, and  hence $\chi_{i+k} + \chi_{n-i +k} = \chi_i + \chi_{n-i}$.
Thus, $\tau_J \tau_{-J} \in B$.
} 

\subsection{Proof of Theorem~\ref{thm:min-const}}
\label{subsec:min-const}

Let $G$ be a vertex- and edge-transitive subgroup of $\Aut(\DC)$, let  $p$ be an odd prime, and let $\wp \colon \Gamma \to \DC$
be a minimal $G$-admissible $p$-elementary abelian covering projection where $\Gamma$ is connected and simple. Since conjugate subgroups lift along isomorphic covering projections \cite{MMP} we may assume, by  Lemma~\ref{lem:AT}, that $G$ is one of the following groups: 

\begin{itemize}
\item[(i) ]
$G=\gen{\rho, \sigma\tau_{\ZZ_n}}$;
\item[(ii) ]
$G=\langle B, \rho\tau_0^\epsilon \rangle$, $1 \neq B \leq K$, $B \normal \Aut_0(\DC)$,  and $\tau_{\ZZ_n}^\epsilon \in B$ where $\epsilon \in \{0,1\}$;
\item[(iii) ]
$G=\langle B, \rho\tau_0^{\epsilon}, \sigma\tau_0^\epsilon\tau_J\rangle$, where $1 \neq B \leq K$,  $B \normal \Aut(\DC)$,  and $\epsilon\in\{0,1\}$ is such that $\tau_{\ZZ_n}^\epsilon \in B$, while $J\subseteq \ZZ_n$ is such that $\tau_J\tau_{-J} \in B$  and $\tau_J\tau_{J+1}\in B$.
\end{itemize}
In each case, let $\generat(G)$ denote the set of the above generators of $G$. As described in Section~\ref{subsec:elemab}, we may assume that the covering projection $\p$ is derived from a voltage assignment on $\DC$ and follow the procedure as described there.

$\bullet$ 
First fix an ordered basis $\{ c_0, c_1, \ldots, c_{n-1},c_*\}$ of $\HG_1(\DC, \ZZ_p)$, where   
\begin{eqnarray*}
c_j  & = & a_j-b_j\in\HG_1(\DC, \ZZ_p), \  j\in \ZZ_n,  \\
c_* & = & a_0+a_1+\ldots+ a_{n-1} + b_0 + b_1+\ldots+ b_{n-1}.
\end{eqnarray*}

$\bullet$ 
Further, for each  $\alpha\in \generat(G)$ we find the matrix $\alpha^\# \in \GL(p,n+1)$ representing the action of $\alpha$ on the homology group  $\HG_1(\DC, \ZZ_p)$  with respect to the above basis. Let
$$
\hat R_\epsilon = (\rho\tau_0^{\epsilon})^{\#}, \quad
\hat S=\sigma^{\#}, \quad
\hat Z = (\sigma\tau_{\ZZ_n})^{\#},\quad
\hat T_X =\tau_X^{\#}.
$$
A straightforward computation shows that:
$$
\hat R_\epsilon =  \sbm{R_{\epsilon} \\ & 1}, \quad 
\hat S=\sbm{S \\ & -1}, \quad
\hat Z = \sbm{-S\\ & -1},\quad
\hat T_X = \sbm{T_X \\ & 1}, 
$$
where 
$$R_{\epsilon} = \begin{bmatrix}
 0  & 1  & &  \\
     &  0 & \ddots         & \\
     &    &\ddots &1 \\
(-1)^ \epsilon& & & 0 \\
  \end{bmatrix},\quad
S=
\sbm{
-1&&&&0\\
&&&&-1\\
&&& -1\\
&&\iddots \\
0&-1
},
$$

$$
T_{X}= \begin{bmatrix}
 (-1)^{x_0} &0&\ldots &0\\
 0& (-1)^{x_1} & & 0\\
 \vdots  &&\ddots &0\\
 0&0&&(-1)^{x_{n-1}}  \\
  \end{bmatrix},$$
with $x_k=1$ if $k\in X$ and $x_k = 0$ otherwise.

$\bullet$ 
In the next step of the procedure we need to find the (minimal) invariant subspaces of the matrix group $G^{\#^t} = \langle (\alpha^{\#})^t : \alpha\in \generat(G)\rangle$. Observe that since $T_X$ and $S$ are symmetric, and $R_{\epsilon}^t = R_{\epsilon}^{-1}$, the matrix groups $G^\# =  \langle \alpha^{\#} : \alpha\in \generat(G)\rangle$ and $G^{\#^t}$ are equal.
For $i\in \ZZ_{n+1}$, let $e_i\in\ZZ_p^{n+1}$ denote the $(i+1)$-th standard basis vector, that is, the vector with $1$ at the
$(i+1)$-st coordinate and $0$ elsewhere. Clearly, the subspaces $E=\gen{e_0,\ldots,e_{n-1}}$ and $\gen{e_n}$ of $\ZZ_p^{n+1}$ are invariant under the action of the matrices $ \hat R, \hat S, \hat Z,\hat T_X$. Hence 
\begin{center}
{\em $\ZZ_p^{n+1}=E\oplus  \gen{e_n}$ is a  $G^\#$-invariant decomposition}.
 \end{center}

$\bullet$ 
Let $W$ be a minimal $G^\#$-invariant subspace. Then 
\begin{center}
 {\em either $W = \la e_n\ra$ or else $W \leq E$}.
\end{center}
Indeed. Suppose that $W\not\leq E$.  Then, by minimality of $W$ we have $W\cap E={0}$,  and $\codim E=1$ implies $\dim W=1$.
So $W=\gen{w}$, where $w$ is a common eigenvector for all elements of $G^{\#}$. Let us consider each  of the groups as in (i), (ii), and (iii) above separately. 

In case (i), where $G=\gen{\rho,\sigma\tau_{\ZZ_n}}$, the vector $w$ is a common eigenvector for $\hat Z$ and $\hat R_0$. We can assume that $w=u+e_n$, where $u\in E$, and we shall denote by $u=(\un{u},0)$ and $w=(\un{u},1)$ the respective row vectors in $\ZZ_p^{n+1}$. The block structure of the matrices then implies 
$$
w\hat R_0=(\un{u},1)\hat R_0=(\un{u}R_0,1),
$$
and so  $\un{u}$ is an eigenvector of $R_0$ with eigenvalue $1$. Similarly, $\lambda w = w \hat Z$ implies 
$$
\lambda (\un{u},1) = (\un{u},1)\hat Z=(-\un{u}S,-1).
$$
It follows that  $\lambda = -1$, and $\un{u}$ is an eigenvector of $S$ with eigenvalue $1$. However, $\un{u}=\un{u}S=\un{u}R_0$ forces $\un{u}=0$, and hence  $W=\gen{e_n}$.

We deal with cases (ii) and (iii) similarly. In both cases, $G^{\#}$ contains a matrix $\hat T_J$ for some $J\not = \emptyset$ and matrix
$\hat R_\epsilon$ for some $\epsilon$. Denoting $w=(\un{u},1)$ as above, equations $w\hat R_{\epsilon}=\lambda w$ and $w\hat T_J =\mu w$ for some $\lambda, \mu \in \ZZ_p$ imply $(\un{u}R_\epsilon,1)=(\un{u},1)$ and $(\un{u}T_J,1)=(\un{u},1)$, so $\un{u}=\un{u}R_\epsilon=\un{u}T_J$, again forcing $\un{u}=0$ and $W=\gen{e_n}$.

$\bullet$
In the last step of the procedure  we find the voltage assignments arising from minimal invariant subspaces as explained  in Section~\ref{subsec:elemab} (see also \cite{MMP}). 
Observe that the voltage assignment $\zeta\colon \HG_1 \to \ZZ_p$ arising from the invariant space $\langle e_n\rangle$ assigns voltage $0$ to every cycle $c_i = a_i -b_i$, forcing the covering graph to be nonsimple, contradicting our assumptions. On the other hand, if $W \leq E$, then $W$ (viewed as a subspace of $\ZZ_p^n$) is  invariant under the multiplication by $R_{\epsilon}$; hence $W$ is an $\epsilon$-cyclic code of length $n$ (see Section~\ref{subsec:codes}). If $g(x)$ is its corresponding generator polynomial, then $W$ is spanned by the rows of the  the matrix $M_{g(x)}$. 
 This completes the proof of Theorem~\ref{thm:min-const}. 
$\hfill\square$

\begin{remark}
\label{rem:degenerate}
{\rm
Note that the generating polynomial $g(x)$  of $W$ is a proper divisor of $\Delta_{n,\epsilon}$ since $W \neq 0$. Also, $W = E$ if and only if  $g(x) = 1$.
} 
\end{remark}

\subsection{Proof of Theorem~\ref{thm:maxG-kernel}}
\label{subec:inv}

Let $n$ be a positive integer, let $\epsilon\in \{0,1\}$, and let $g(x)=\alpha_0+\alpha_1x+\ldots + \alpha_m x^m$
be a proper divisor of $\Delta_{n,\epsilon}(x) \in \ZZ_p[x]$. Consider the covering projection 
$$
\wp_{g(x)}\colon \Gamma_{g(x)} \to \DC.
$$

As in Subsection~\ref{subsec:min-const},  let $e_0, \ldots, e_n$ be the standard basis  of $\ZZ_p^{n+1}$.   Further, let $W \leq E =\la e_0, \ldots, e_{n-1}\ra$ be the subspace  spanned by the rows of the matrix $M_{g(x)}$, and let $W^*$ be the subspace of $\ZZ_p^{n+1}$  obtained from  $W$ by adding a $0$ in the last component.  Then $W^*=W \oplus \langle \underline{0} \rangle$.

\medskip
\noindent
{\sc Claim 1.}
{\em An automorphism $\alpha$ of $\DC$ lifts along $\wp_{g(x)}$ if and only if the restriction $(\alpha^{\#})^t|_E$ of $(\alpha^{\#})^t$ to $E$ leaves $W$ invariant.}

\medskip
\noindent
{\em SubProof}.
From Definition~\ref{def:matrix} it easily follows that  the voltages of  the base vectors $\{c_0, c_1,\ldots, c_{n-1},c_*\}$ of $H_1 =  \HG_1(\DC,\ZZ_p)$ correspond to the columns of the matrix $[2M_{g(x)}|\underline{0}]$, obtained from $M_{g(x)}$ by multiplying all its elements by $2$ and adding an extra zero column. Since $p$ is an odd prime, the rows of $2M_{g(x)}$ and $M_{g(x)}$ span the same space. Hence $W^* = W \oplus \underline{0}$ is spanned by the rows of $[2M_{g(x)}|\underline{0}]$.

Following the explanation in Subsection~\ref{subsec:elemab}, an automorphism $\alpha$ of $\DC$ lifts along $\wp_{g(x)}$ if and only if the subspace spanned by the rows of $[2M_{g(x)}|\underline{0}]$ is $(\alpha^\#)^t$-invariant, that is, if and only if $W^*$ is  $(\alpha^\#)^t$-invariant. However,  as observed in Section~\ref{subsec:min-const},  $(\alpha^\#)^t$ preserves the decomposition $E\oplus \la e_n \ra$ of $\ZZ_p^{n+1}$, which implies that  $W^*$ is $(\alpha^\#)^t$-invariant if and only if $W$ is $(\alpha^\#)^t|_E$-invariant. This proves Claim 1. 
$\hfill\square$

\bigskip
We now turn to proving Theorem~\ref{thm:maxG-kernel}.
The space $E\cong \ZZ_p^n$ is conveniently viewed as the polynomial ring  $\ZZ_p[x]/(\Delta_{n,\epsilon}(x))$ by identifying $e_i \in E$ with the polynomial $x^i \in \ZZ_p[x]/(\Delta_{n,\epsilon}(x))$. In view of this identification we shall think of matrices  $R_\epsilon$, $T_J$, $S$ and $Z$ defined in Subsection~\ref{subsec:min-const} as acting on the corresponding polynomials.
Let $G$  be the maximal group that lifts along the covering projection $\wp_{g(x)}\colon \Gamma_{g(x)} \to \DC$. We need to show the following:
\begin{itemize}
\item[(A)]  $G$ is vertex- and edge-transitive, with $B = G \cap K $ nontrivial;
\item[(B)] 
the kernel of the action of $G$ on the vertex set of $\DC$  is isomorphic to  $\ZZ_2^d$, where $1\leq d \leq n$ is the maximal integer $d =\Expn(g(x))$ such that $g(x) = g_d(x^d)$. 
\comment{
(Recall that $d$ is a divisor of $n$, by  Lemma~\ref{lem:d-divides-n}, and the associated polynomial  $g_d(x)$ is a proper divisor of $\Delta_{n/d,\epsilon}(x) \in \ZZ_p[x]$.);
} 
\item[(C)] 
$G$ is arc-transitive if and only if  $g(x)$ is weakly reflexible, that is, if and only if $g_d(x)$ is reflexible. In particlar, $\Aut(\DC)$ lifts if and only if $\rho$ and $\tau_0$ lift if and only if $g(x) = 1$;
\item[(D)]
if $g(x)$ is not weakly reflexible, then $\wp_{g(x)}$ is minimal if and only if $g_d(x)$ is a maximal  divisor of $\Delta_{n/d,\epsilon}$. 
If $g(x)$ is weakly reflexible, then $\wp_{g(x)}\colon \Gamma_{g(x)} \to \DC$    is minimal $G$-admissible if and only $g_d(x)$ is a maximal weakly reflexible divisor of $\Delta_{n/d,\epsilon}$. In particular, the covering is minimal $\Aut(\DC)$-admissible
 if and only if  $g(x) = 1$.   
\end{itemize}

\medskip
\noindent
{\sc Claim 2.} 
{\em $G$ contains $\rho\tau_0^\epsilon$ and $\tau_{\ZZ_n}$; in particular, $G$ is vertex- and edge-transitive.}

\medskip
\noindent
{\em SubProof}.
 In view of the identification of $E\cong \ZZ_p^n$ with $\ZZ_p[x]/(\Delta_{n,\epsilon}(x))$, the subspace $W$   corresponds to an $\epsilon$-cyclic code with the generating polynomial $g(x)$. In particular, $W$ is invariant under $R_\epsilon$.   Since $((\rho\tau_0^\epsilon)^\#)^t = (R_\epsilon)^t = R_\epsilon^{-1}$ it follows by    Claim~1 that $\rho\tau_0^\epsilon$ lifts.  Similarly, the subspace $W$ is invariant under  $T_{\ZZ_n} = -I$, and therefore $\tau_{\ZZ_n}$ lifts.  This proves Claim 2 and hence part (A).
$\hfill\square$

\medskip
\noindent
{\sc Claim 3.}
{ \em The kernel $B = G \cap K$ of the action of $G$ on the vertex set of $\DC$ is equal to $B= \langle   \tau_{[0,d]}, \tau_{[1,d]}, \ldots \tau_{[d-1,d]} \rangle \cong \ZZ_2^d$, where $1\leq d \leq n$ is the maximal integer such that $g(x) = g_d(x^d)$.}

\medskip
\noindent
{\em SubProof}.
Note that after identifying $E$ with $\ZZ_p[x]/(\Delta_{n,\epsilon}(x))$ as described above, the subspace $W$ is generated by polynomials $x^j g(x)$, $j\in \{0,\ldots,n-1\}$. Since 
$g(x)=g_d(x^d) =\alpha_0+\alpha_dx^d+\ldots + \alpha_{td} x^{td}$, $m = td$, the action of $T_{[i,d]} = \tau_{[i,d]}^\#$ on the above   generating set of $W$ is given by 
\begin{equation}\label{eq:act-B}
(x^j g(x)) T_{[i,d]} = \left\{ \begin{array}{rl} x^j g(x) & \hbox{ if } d \nmid |i-j|\\  - x^j g(x) & \hbox{ if } d \mid |i-j| \end{array} \right. .
\end{equation}
Therefore, $W$ is $T_{[i,d]}$-invariant for each $i$. By Claim~1,   $\tau_{[i,d]}$ lifts and so  $\langle   \tau_{[0,d]}, \tau_{[1,d]}, \ldots \tau_{[d-1,d]} \rangle \le B$.

Let $k$ be the exact period of $B$. We have just seen that $B$ contains the largest subgroup of $K$ with period $d$. But overgroups cannot have smaller periods, hence  $k\ge d$. In order to prove that  $B$ is in fact equal to  $\langle   \tau_{[0,d]}, \tau_{[1,d]}, \ldots \tau_{[d-1,d]}\rangle$ we need to show that  $k=d$. To this end we prove that $g(x) = h(x^k)$. 

Let $0 < s <  n$ be an integer not divisible by $k$. Then $s$ is not a period of $B$, and so   there exists $\tau_L \in B$ such that $\tau_L \neq \tau_{L+s}$. This means that there is $i \in L$ such that $i + s \not\in L$. Now $B$ is clearly normal in $G$. Since $G$  contains 
$\rho\tau_0^{\epsilon}$, by Claim~2,  it follows that $B$ is normalized by $\rho$. Thus, replacing $\tau_L$ with its conjugate by an appropriate power of $\rho$ we find that $B$ contains an element $\tau_L$ such that $0 \in L$ and $s \not\in L$. Then $(g(x))T_L+g(x)$ is a polynomial of degree at most that of $g(x)$ and with a zero constant term. By applying an appropriate negative power of $R_\epsilon$ to it, one obtains an element of $W$ with degree strictly less than that of $g(x)$. Since $g(x)$ is the generator of the $\epsilon$-cyclic code $W$ we must have $(g(x))T_L+g(x) = 0$. Now $s\not\in L$ and $p$  odd together  imply that $\alpha_s = 0$. But $s$ was an arbitrary positive integer not divisible by $k$, so $g(x) = h(x^k)$ for some polynomial $h(x)$. By the choice of $d$ we have  $k\le d$. But we already know that $k\ge d$,  so $k=d$. This proves Claim~3 and hence part (B).
$\hfill\square$

\medskip
\noindent
{\sc Claim 4.}
{ \em 
$G$ is arc transitive if and only if $g(x)$ is weakly reflexible (that is,  if and only if $g_d(x)$ is reflexible). Inparticular, $\Aut(\DC)$ lifts if and only if $\rho$ and $\tau_0$ lift if and only if $g(x) = 1$.
}

\medskip
\noindent
{\em SubProof}.
By Claims~2 and 3 we already know that the group $G$ contains $\rho\tau_0^{\epsilon}$ and $B=\langle   \tau_{[0]}, \tau_{[1]}, \ldots \tau_{[d-1]} \rangle$. By Lemma~\ref{lem:AT}, the group $G$ is arc transitive if and only if it contains $\sigma\tau_0^{\epsilon}\tau_J$
for some $J \subseteq \ZZ_n$ satisfying $\tau_J\tau_{J+1} \in B$, $\tau_J\tau_{-J}\in B$. It remains to show that 
{\em $\sigma\tau_0^{\epsilon}\tau_J$ lifts if and only if  $g_d(x)$ is reflexible.}
Let  $\tau_C \in B$ be arbitrary and set $\tau_L = \tau_J\tau_C$. Then  $\sigma\tau_0^{\epsilon}\tau_J$ lifts if and only if $\sigma\tau_0^{\epsilon}\tau_L$ lifts. Moreover, $\tau_L\tau_{L+1} = \tau_J\tau_{J+1}\tau_C\tau_{C+1} \in B$ since $\tau_J\tau_{J+1} \in B$ and $\tau_C, \tau_{C+1} \in B$. Similarly, $\tau_L\tau_{-L} = \tau_J\tau_{-J} \tau_C\tau_{-C} \in B$.  Instead of $J$ we work with a suitably chosen  $L$ (which will be defined later),  and  show that

\begin{center}
{\em $\sigma\tau_0^{\epsilon}\tau_L$ lifts if and only if $g_d(x)$ is reflexible}.
\end{center}

As above, we shall think of $W$ as an ideal in $\ZZ_p/(\Delta_{n,\epsilon}(x))$ generated by the polynomial $g(x)$. By Claim~1 we know that $\sigma\tau_0^{\epsilon}\tau_L$ lifts if and only if $ST_0^{\epsilon} T_J$ preserves  $W$. We also know that $W$ is invariant under the actions of $B$ and  $R_\epsilon = RT_0^\epsilon$. We now show that in order for $W$ to be $ST_0^{\epsilon} T_L$-invariant it is enough to require that  the mapped generating polynomial stays in $W$: 
\begin{center}
{\em $W$ is  $ST_0^{\epsilon} T_L$-invariant  if and only if  $g(x) ST_0^{\epsilon} T_L \in W$}.
\end{center} 
Indeed. Clearly, if $ST_0^{\epsilon} T_L$  preserves $W$  then  $g(x) ST_0^{\epsilon} T_L \in W$. For the converse,  observe (by induction) that
$$
R_\epsilon^i (ST_0^\epsilon T_L) = (ST_0^\epsilon T_L) (T_L T_{L+1} R_\epsilon^{-1})^i.
$$
Since $\tau_L\tau_{L+1} \in B$ and $\rho \tau_0^{\epsilon}\in G$, we have  that $(T_LT_{L+1} R_\epsilon^{-1})^i$ preserves $W$. 
Hence if $g(x) (ST_0^{\epsilon} T_L) \in W$, then
$$
 x^i g(x) (ST_0^\epsilon T_L) =  (g(x) R_\epsilon^i) (ST_0^\epsilon T_L) = g(x) (ST_0^\epsilon T_L) (T_LT_{L+1} R_\epsilon^{-1})^i \in W.
$$
In particular, $W$ is $ST_0^\epsilon T_L$-invariant, as required. It therefore remains to show that 
\begin{center}
{\em $g(x) (ST_0^\epsilon T_L) \in W$ if and only if $g_d(x)$ is reflexible}.
\end{center}
To this end let us write the matrix $T_L$ as $T_L = \text{diag}(-\psi_0, -\psi_1, \ldots, -\psi_{n-1})$, where  
$$
\psi_i = \left\{\begin{array}{rl}
                         1 & i  \in L \\
                        -1 & i  \notin L.
                  \end{array}\right.
$$
By computation we obtain that $ST_0^\epsilon T_L$ acts on 
$g(x)=\alpha_0+\alpha_1x+\ldots + \alpha_m x^m$  by the rule 
\begin{align*}
 g(x) (ST_0^\epsilon T_L) \> & =  (-1)^\epsilon\psi_0\alpha_0 + \psi_{n-m}\alpha_mx^{n-m} + \ldots + \psi_{n-1}\alpha_1 x^{n-1} \\
 & =   ( \psi_{n-m}\alpha_m + \ldots + \psi_{n-1}\alpha_1 x^{m-1}  + \psi_0\alpha_0 x^{m} )\, R_\epsilon^{n-m}.
\end{align*}
Since $W$ is invariant under the action of $R_\epsilon$, it follows that $ g(x) (ST_0^\epsilon T_L) \in W$ if and only if $\psi_{n-m}\alpha_m + \ldots + \psi_{n-1}\alpha_1 x^{m-1}  + \psi_0\alpha_0 x^{m} \in W$. But this  holds if and only if there exists $\lambda \in \ZZ_p^*$ such that
$ \lambda g(x) = \psi_{n-m}\alpha_m + \ldots + \psi_{n-1}\alpha_1 x^{m-1}  + \psi_0\alpha_0 x^{m}$.
Finally, this last requirement  is equivalent to 
\begin{equation} \label{lambda}
\lambda \alpha_{m-i}=\psi_{n-i}\alpha_{i},\quad i=0,\ldots, m,
\end{equation}
with understanding that $\psi_n=\psi_0$. Since $g(x) = g_d(x^d) = \alpha_0 + \alpha_d x^d + \ldots + \alpha_{td} x^{td}$,  condition (\ref{lambda}) is relevant only for the coefficients of the form $\alpha_j' =\alpha_{jd}$, that is, for the coefficients of the polynomial $g_d(x) = \alpha_0' + \alpha_1' x + \ldots + \alpha_t' x^t$. So (\ref{lambda}) rewrites as 
\begin{equation} \label{lambda0}
\lambda \alpha_{t-j}' = \psi_{n-j d}\alpha_{j}',\quad j=0,\ldots, t.
\end{equation}
To summarize,
\begin{center}
{\em  $\sigma\tau_0^{\epsilon}\tau_L$ lifts if and only if   (\ref{lambda0}) holds.}
\end{center}

At this point we  make use of  Lemma~\ref{lem:special-J}. Since $B$ is the largest subgroup with exact period $d$ we can take $L$ to have  the special structure as described there. If $L=\ZZ_n$ we have $\psi_i = 1$ for all $i$, and   condition (\ref{lambda0}) becomes  
$\lambda \alpha_{t-j}' = \alpha_{j}'$,  for $j=0,\ldots, t$, that is,  $g_d(x)$ must be type-1 reflexible (and then $g(x)$ is type-1 reflexible as well, by Lemma~\ref{lem:refl-unique}).
Note that $\lambda \alpha_m=\alpha_0$ and $\lambda \alpha_0=\alpha_m$ imply $\lambda=\pm 1$. 

Otherwise, if $L \neq \ZZ_n$ then  $n$ must be even and  divisible by  $2d$.  In view of the structure of $L$ we now  have $\psi_i = (-1)^{(i- i\,\hbox{{\rm\small mod}}\, d)/d}$. In particular, if $d \mid i$ then $\psi_{n-i} = \psi_i = (-1)^{i/d}$ (while if $d \nmid i$ then $\psi_{n-i} = - \psi_i$). For  $d \mid i$ condition (\ref{lambda}) rewrites as  $\lambda \alpha_{m-i}=(-1)^{i/d}\alpha_i$ and hence (\ref{lambda0}) rewrites as $\lambda \alpha_{t-j}' = (-1)^j \alpha_j'$. Thus,  $g_d(x)$ is  type-2 reflexible (by Lemma~\ref{lem:refl-unique}, $g(x)$ need not be reflexible, that is, $g(x)$ is only weakly reflexible).
Note that  $\lambda \alpha_m = \alpha_0$ and $\lambda \alpha_0 = (-1)^{m/d}\alpha_m$ imply $\lambda^2 = (-1)^{m/d}=\pm 1$. 

In particular, let us prove that {\em $\Aut(\DC)$ lifts if and only if $\rho$ and $\tau_0$ lift if and only if $g(x) = 1$}. Indeed. If $\Aut(\DC)$ lifts then $\rho$ and $\tau_0$ lift. If
$\rho$ and $\tau_0$ lift, then $K = \la \tau_0^{\rho^i}\ |\ i \in \ZZ_n\ra$ lifts; so $B = K$ and $d = n$, which implies $g(x) = 1$. If $g(x) = 1$, then $W = E$ is invariant for $\Aut(\DC)^{\#}$, and so $\Aut(\DC)$ lifts. 

This proves Claim~4 and hence part (C).
$\hfill\square$

\begin{remark}
\label{rem:tauL-reflextype}
{\rm 
We have just seen that  the reflexibility type of  $g_d(x)$ is uniquely determined by $\tau_L$. But the implication holds in the other direction as well: $\tau_L$ is uniquely determined by the reflexibility type of $g_d(x)$. Indeed. If $g_d(x)$ is type-1 reflexible, then $L = \ZZ_n$ for otherwise $g_d(x)$ would be type-2 reflexible; but this is a contradiction, by   Lemma~\ref{lem:refl-unique}, since $\Expn(g_d(x)) = 1$. Similarly, if 
$g_d(x)$ is type-2 reflexible, then $L \neq \ZZ_n$ for otherwise $g_d(x)$ would be type-1 reflexible, and we get the same contradiction as above.
} 
\end{remark}

\medskip
\noindent
{\sc Claim 5.}
{ \em 
If $g(x)$ is not weakly reflexible, then the covering projection $\wp_{g(x)}\colon \Gamma_{g(x)} \to \DC$    is minimal $G$-admissible  if and only $g_d(x)$ is a maximal  divisor of $\Delta_{n/d,\epsilon}$. Otherwise, if  $g(x)$ is weakly reflexible, then 
$\wp_{g(x)}$ is minimal if and only if $g_d(x)$ is a maximal weakly reflexible divisor of $\Delta_{n/d,\epsilon}$. In particular, the covering is minimal $\Aut(\DC)$-admissible
 if and only if  $g(x) = 1$. 
}

\medskip
\noindent
{\em SubProof}.
Let us first consider the case when $g(x) = 1$. We have already proved that $\Aut(\DC)$ lifts if and only if $g(x) = 1$. This means that the covering arising from $g(x) = 1$ is minimal $\Aut(\DC)$-admissible. Now the associated polynomial of $g(x) = 1$ is $g_n(x) = 1$, and $g_n(x) = 1$ is a maximal weakly reflexible divisor of $\Delta_{1,\epsilon} = x - (-1)^{\epsilon}$, as claimed.

Consider the case when  $g(x)$ is not weakly reflexible, that is, the maximal group $G$ that lifts along  $\wp_{g(x)}\colon \Gamma_{g(x)} \to \DC$ is not arc transitive, by part (C), and we may therefore assume that $G = \la B, \rho\tau_0^{\epsilon}\ra$.  Also recall that the covering is minimal whenever $G$ does not lift along a `smaller nontrivial covering', which amounts to saying that the respective nontrivial invariant subspace -- the $\epsilon$-cyclic code  arising from  $g(x)$ -- is  a minimal nontrivial $G^\#$-invariant subspace. In terms of ideals this boils down to requiring that $(g(x))$ is a minimal nontrivial $B^{\#}$-invariant ideal in $\ZZ_p[x]/(\Delta_{n,\epsilon}(x))$.  So what we need to show is that 

\begin{center}
{\em 
$(g(x))$ is a minimal nontrivial $B^{\#}$-invariant ideal in $\ZZ_p[x]/(\Delta_{n,\epsilon}(x))$\\ if and only if $g_d(x)$ is a maximal divisor of $\Delta_{n/d, \epsilon}(x)$.
}
\end{center}
Note in particular that $(g(x))$ is a proper ideal since $g(x) \neq 1$. 
First suppose that the ideal $(g(x))$ is not minimal. Then it properly contains a nontrivial $B^{\#}$-invariant ideal, say $0 <(q(x)) < (g(x))$. Hence   $q(x) \neq \Delta_{n, \epsilon}(x)$, and $q(x) = g(x)\,h(x)$ where $h(x)$ is not a constant polynomial. Let $d' = \Expn(q(x))$ be the maximal integer such that $q(x) = q_{d'}(x^{d'})$, and let
 $G'$ be the maximal group that lifts along the respective covering associated with $(q(x))$. By part (B) the integer  $d'$ is the exact period of $B'= G' \cap K$.  
Now, since $(q(x))$ is  $B^{\#}$-invariant we have $G' \geq G$, which implies that  $G' \cap K = B' \geq B$. Therefore  $d'$ is also a period for $B$ and hence $d'$ is divisible by $d$. Consequently, $q(x) = q_{d'}(x^{d'})$ is a polynomial in $x^d$, say 
$Q(x^d) = q(x) = g_d(x^d)\,h(x)$. By Lemma~\ref{lem:poly-in-d} we have 
$h(x) = k(x^d)$, and so 
$$Q(x) = g_d(x) k(x).$$
 Now $k(x)$ is nonconstant since $h(x)$ is nonconstant,  and since $q(x) \neq \Delta_{n, \epsilon}(x)$ we have that $Q(x) \neq \Delta_{n/d, \epsilon}(x)$. Thus,  $g_d(x)$ is not a maximal divisor of $\Delta_{n/d,\epsilon}(x)$.

Suppose now that $g_d(x)$ is not a maximal divisor of $\Delta_{n/d, \epsilon}(x)$, that is,
there is a nonconstant polynomial  $h(x)$ such that $g_d(x) h(x)$ is a proper divisor of $\Delta_{n/d,\epsilon}(x)$. Then 
$$q(x) = g(x)\,h(x^d)$$
 is proper divisor of $\Delta_{n,\epsilon}$, and $(q(x))$ is a nontrivial proper sub-ideal of $(g(x))$. 
The respective subspace is generated by the polynomials $x^j q(x)$, $j\in \{0,\ldots,n-1\}$.  Similarly as in  (\ref{eq:act-B}),  since $q(x)$ is a polynomial in $x^d$ the action of the generators $T_{[i,d]} = \tau_{[i,d]}^\#$ of $B^{\#}$ on the above generating set shows that the ideal  $(q(x))$ is $B^{\#}$-invariant.

\bigskip
It remains to consider the case when the polynomial $g(x)$ is weakly reflexible, that is, the maximal group $G$ that lifts along  $\wp_{g(x)}\colon \Gamma_{g(x)} \to \DC$ is arc transitive and hence of the form $G = \la B, \rho\tau_0^{\epsilon}, \sigma\tau_0^{\epsilon} \tau_L \ra$. In terms of ideals, the respective covering is minimal if and only if $(g(x))$ is a minimal nontrivial $\la B, \sigma\tau_0^{\epsilon} \tau_L \ra^{\#}$-invariant ideal. We therefore need to prove that 

\begin{center}
{\em
$(g(x))$ is a minimal nontrivial $\la B, \sigma\tau_0^{\epsilon} \tau_L \ra^{\#}$-invariant ideal \\ if and only if $g_d(x)$ is a maximal weakly reflexible divisor of $\Delta_{n/d,\epsilon}(x)$. 
}
\end{center}
The case $g(x) = 1$ has already been considered above, so we may assume $g(x) \neq 1$.
Suppose that the ideal $(g(x))$ is not  minimal.  Then it must properly contain some nontrivial  $\la B, \sigma\tau_0^{\epsilon} \tau_L \ra^\#$-invariant ideal, say $0 < (q(x)) < (g(x))$.  Consequently,   there exists a nonconstant polynomial $h(x)$ such that 
$q(x) = g(x) h(x)$ is a proper divisor of $\Delta_{n,\epsilon}(x)$.  Let $G'$ be the maximal group that lifts along the covering associated with $(q(x))$. 
By part (B) the integer  $d'  = \Expn(q(x))$ is the exact period of $B'= G' \cap K$.
But since $(q(x))$ is $\la B, \sigma\tau_0^{\epsilon} \tau_L \ra^\#$-invariant we have
$G \leq G'$. Hence $B \leq B'=G'\cap K$,  and so $d'$  is a period for $B$. Thus, $d \mid d'$. Let  
$$Q(x) = q_{d'}(x^{d'/d}).$$
 Then $Q(x^d) =q_{d'}(x^{d'}) = g_d(x^d) h(x)$.
By Lemma~\ref{lem:poly-in-d} we have that $h(x) = k(x^d)$, and so 
$$Q(x) = g_d(x)k(x).$$
 Let $d'' = \Expn(Q(x))$. By Lemma~\ref{lem: Q-ofx} it follows that 
$Q_{d''}(x) = q_{d'}(x)$. Now $q_{d'}(x)$ is reflexible, by part (C),  since $G' \geq G$ is arc transitive. Hence $Q_{d''}(x)$ is reflexible which means that $Q(x)$ is weakly reflexible. But then $g_d(x)$ (although reflexible)  is not  a maximal  weakly reflexible divisor of $\Delta_{n/d,\epsilon}(x)$.

To prove the claim in the other direction, suppose that  
$g_d(x)$ (although reflexible) is not a maximal weakly reflexible proper divisor of $\Delta_{n/d,\epsilon}(x)$ (in particular, $g_d(x)$ is not a maximal proper divisor since a reflexible polynomial is also weakly reflexible).
This means that there exists a weakly reflexible polynomial  
$$q(x) = g_d(x) h(x),$$
 where $h(x)$ is nonconstant and $q(x)$ properly divides $\Delta_{n/d,\epsilon}(x)$. With
$d'' = \Expn(q(x))$ we have $q(x) = q_{d''}(x^{d''})$ and the polynomial $q_{d''}(x)$ is reflexible. Now set 
$$Q (x) = q(x^d) = g(x) h(x^d),$$
 and let $d'= \Expn(Q(x))$. By  Lemma~\ref{lem:Expn-gcd} we have that $d \mid d'$. Next, from $Q(x) = q(x^d)$ we have $Q_{d'}(x^{d'}) = q_{d''}(x^{d d''})$, and so $d d'' \mid d'$. Moreover, since $d \mid d'$ we have $Q_{d'}(x^{d'/d}) = q(x) = q_{d''}(x^{d''})$, and hence $d'/d$  divides $d''$.
Consequently, $d d'' = d'$, and $Q_{d'}(x^{d'}) = q_{d''}(x^{d'})$. Thus, 
$$Q_{d'}(x) = q_{d''}(x).$$
 Since $q_{d''}(x)$ is reflexible, $Q_{d'}(x)$ is reflexible. Therefore, the largest subgroup $G'$ that lifts along the covering associated with the ideal $(Q(x))$ is arc transitive, by part (C). We now show that the ideal $(Q(x))$ is $\la B, \sigma\tau_0^{\epsilon} \tau_L \ra^\#$-invariant. The respective subspace 
is generated by the polynomials $x^j Q(x)$, $j\in \{0,\ldots,n-1\}$.  
Similarly as in  (\ref{eq:act-B}),  since $Q(x)$ is a polynomial in $x^d$, the action of the generators $T_{[i,d]} = \tau_{[i,d]}^\#$ of $B^{\#}$ on the above generating set of polynomials shows that the ideal  $(Q(x))$ is $B^{\#}$-invariant. Therefore
$B \leq B' = G' \cap K$. Hence  $\la B, \rho\tau_0^{\epsilon}\ra \leq  \la B', \rho\tau_0^{\epsilon}\ra$. So the index-$2$ subgroup $G \cap \Aut_0(\DC)$ of $G$  is contained in the index-$2$ subgroup $G' \cap \Aut_0(\DC)$ of $G'$, which implies  $G \leq G'$. This shows that the nontrivial ideal $(Q(x))$ is $\la B, \sigma\tau_0^{\epsilon} \tau_L \ra^{\#}$-invariant, as required. 
 
\medskip
This proves Claim~5, that is,  part (D),  and completes the proof of Theorem~\ref{thm:maxG-kernel}.
$\hfill\square$

\subsection{Proof of Theorem~\ref{thm:main}}
\label{subsec:main-proof}

Theorem~\ref{thm:main} is a simple corollary of Theorems~\ref{thm:maxG-kernel} and \ref{thm:min-const}. Indeed. Let $\Gamma$ be a $4$-valent graph and $H \leq \Aut(\Gamma)$ a group of its automorphisms acting transitively on vertices and edges.  Further, let  $N \cong \ZZ_p^r$,  where $p$ is an odd prime, be  a minimal normal subgroup of $H$ such that the simple quotient $\Gamma_N$ is a cycle $C_n$, $n\geq 3$. 

Then, since $H$ is edge-transitive and the orbits of $N$ are blocks of imprimitivity for $H$, no vertex  in $\Gamma$  has exactly one neighbour in one `adjacent  orbit' and three neighbours in the `other adjacent orbit'. Also, no vertex has four adjacent neighbours in one `adjacent orbit' since $\Gamma_N$ is a cycle of length $n \geq 3$.   Thus, the graph induced between two `adjacent orbits' is a collection of cycles of even lengths, and hence a collection of two perfect matchings.  Since $N$ is of odd order,  each orbit contains an odd number of vertices (a power of $p$). Suppose now that some element  $\alpha \in N$ fixes a vertex.  Because $N$ is abelian,  $\alpha$ fixes all vertices in that orbit. But then $\alpha$ fixes all vertices in both `adjacent orbits'.  For if there was no fixed vertex 
 in an `adjacent orbit', the vertices in that orbit would be switched in pairs by the action of $\alpha$ --  which is impossible since each orbit has odd length. It follows that the action of $N$ is semiregular on vertices, and so  the quotient projection $q_N\colon \Gamma \to \Gamma/N \cong \DC$ is a regular covering projection, by Lemma~\ref{lem:semiregular}. 

Since $N$ is a minimal normal subgroup of $H$,  the projection $q_N$ is minimal in the sense that $H/N$ does not lift   along a `smaller regular covering'.  By Theorem~\ref{thm:min-const}, $q_N$ is isomorphic to a regular covering projection  $\wp_{g(x)} \colon \Gamma_{g(x)} \to \DC$, where $g(x)$ is a proper divisor of $\Delta_{n,\epsilon}(x)$, $\epsilon \in \{0,1\}$. Since the voltage group associated with this covering is $\ZZ_p^r$, we have $r \leq n$ and $g(x)$ has degree $m = \deg(g(x)) = n-r$. 
Moreover, if $d = \Expn(g(x))$ is the maximal integer such that $g(x) = g_d(x^d)$, then $d$ is a divisor of $n$, by Lemma~\ref{lem:d-divides-n}, and so $g_d(x)$ is a proper divisor of $\Delta_{n/d,\epsilon}(x)$. Also, since $d$ divides $m = n-r$ it also divides $r$. 

Recall that the largest group that lifts along $\wp_{g(x)}$ lifts to the normalizer $\N_{\Aut(\Gamma)}(N)$ of  $N$ within $\Aut(\Gamma)$, and the largest group that lifts and the normalizer $\N_{\Aut(\Gamma)}(N)$ are simultaneously arc transitive or not. By Theorem~\ref{thm:maxG-kernel} it  follows  that $\N_{\Aut(\Gamma)}(N)$ is arc transitive if and only if $g(x)$ is weakly reflexible. Also by Theorem~\ref{thm:maxG-kernel} -- since $\wp_{g(x)}$ is a minimal covering projection -- the fact whether the normalizer $\N_{\Aut(\Gamma)}(N)$ is arc transitive or not depends on whether 
the associated polynomial $g_d(x)$ is either  a maximal weakly reflexible divisor of $\Delta_{n/d,\epsilon}(x)$ or else a maximal  divisor of  $\Delta_{n/d,\epsilon}(x)$.

Finally, since $q_N$ is a regular covering projection, the vertex stabilizer  of  $\N_{\Aut(\Gamma)}(N)$ projects isomorphically onto the stabilizer of a vertex within the maximal group $G$ that lifts. This stabilizer is either $B \cong \ZZ_2^d$ when $G$ is not arc transitive, or else $\la\sigma'\ra \ltimes B \cong 
\ZZ_2 \ltimes \ZZ_2^d$  when $G$ is arc transitive, where $\sigma'$ is the reflexion of $\DC$ fixing a vertex. This concludes the proof. 
$\hfill\square$

\section{Further examples}
\label{sec:GardPra}

We conclude the paper with two further examples of our construction.
In particular, in Example \ref{ex:extr} we consider the extremal case with graphs $C^{\pm 1}(p;rt,r)$ and $C^{\pm \theta}(p;2rt,r)$ from Theorem \ref{GaP1}. 


\begin{example}
For an integer $n \geq 3$ and $p$ an odd prime, the polynomial $g(x)=1+x+x^2+\ldots+x^{n-1}\in\ZZ_p[x]$ is a maximal divisor of $x^n-1$. Its associated  matrix is $M_{g(x)}=\sbm{1&1&1&\ldots&1}$.
The  graph $\Gamma_{g(x)}$ is a  tetravalent graph  on the vertex set $\ZZ_p\times \ZZ_n$ with adjacency relations $(v,j)\sim(v\pm 1,j+1)$,  and is isomorphic to the tensor product of cycles $C_p\times C_n$. 
Now  $g(x) = g_d(x)$ with $d=1$ is a maximal divisor of $x^n - 1$ and is reflexible. Hence the cover is minimal, and the  maximal group $G$ that lifts is arc transitive. Here $B = \la \tau_{\ZZ_n}\ra $ and  $\epsilon=0$. 
Since $g_d(x) = g(x)$ has reflexibility type-$1$ we have  $\tau_L = \tau_{\ZZ_n}$, see
Remark~\ref{rem:tauL-reflextype}. Therefore $G$  is isomorphic to $\gen{\tau_{\ZZ_n},\rho,\sigma}\cong \ZZ_2\ltimes \Dih(n)$. 

In similar fashion, if $n$ is odd, then $g(x)=1-x+x^2-\ldots+x^{n-1}$ is a maximal divisor of $x^n+1$, the  associated matrix is $M_{g(x)}=\sbm{1&-1&1&\ldots&1}$, but due to $\pm$ symbol in the definition of $\Gamma_{g(x)}$  we get the same graph.
\end{example}

\begin{example}\label{ex:extr}
For a given integer $n \geq 3$, let us consider the extremal case where the stabilizer 
$|M_v|$ as in Theorem~\ref{GaP1} is maximal. This is equivalent to requiring that $d =\Expn(g(x)) = r$.
Let  $s = n/d = n/r$, so $\Delta_{n/d, \epsilon} = x^s - (-1)^{\epsilon}$, and  
let  $\deg\,g_d(x) = t$. From $\deg\,g(x) = dt =  rt = n - r = rs -r$ we get $t = s -1$. Therefore,  
\begin{center}
{\em the extremal case arises if and only if $g_d(x)$ generates a $1$-dimensional cyclic or negacyclic code in $\ZZ_p^s$, for some divisor  $s \mid n$, where $g_d(x)$ is reflexible}. 
\end{center}
Note that   the full automorphism group lifts if and only if $d = r = n$, or equivalently,  if and only if $\deg(g(x)) = 0$, and this happens if and only of $t = 0$, $s = 1$. This case is somewhat special since $g(x) = g_d(x) = 1$. The respective graph $\Gamma_{g(x)}$  is obtained from $M_{g(x)}= I \in \ZZ_p^{n \times n}$.  In our further analysis we assume that $s >1$.

\bigskip
Let $\theta \in \ZZ_p^{*}$ be such that $\theta^s = (-1)^{\epsilon}$. Then 
$\Delta_{n/d, \epsilon}  = (-1)^{\epsilon}(\theta x-1) (1 + \theta x + \ldots + \theta^{s-1} x^{s-1})$, and so $g_d(x) = 1 + \theta x + \ldots + \theta^{s-1} x^{s-1}$.
In order to determine  all those $\theta$ for which  $g_d(x)$ is reflexible we split the analysis into two  cases according to the type of reflexibility.  

\bigskip
$\bullet$ \noindent{\sc Case 1}.
{\em Let $g_d(x)$ be reflexible of type $1$}: there exists $\lambda \in \ZZ_p^{*}$ (which must be $\pm 1$) such that  $\lambda \theta^{s-1-j} = \theta^j$ for all $j$. In particular,  
$\lambda = \theta^{s-1}$.   Hence  $\theta =(-1)^{\epsilon +1}$, and so $\theta = \pm 1$.  If $\theta= 1$ then $1^s = (-1)^{\epsilon}$ forces $\epsilon = 0$, and $g_d(x) = 1 + x + \ldots + x^{s-1}$ arises from the factorization of $x^s-1$; there are no other restrictions on $s$.  Let $\theta = -1$. Then the equality $(-1)^s = (-1)^{\epsilon}$ forces $s$ to be even for $\epsilon = 0$, and odd for $\epsilon =1$. 
So $g_d(x) = 1 - x + x^2 - \ldots + (-1)^{s-1} x^{s-1}$ arises either from
$x^s - 1$ or $x^s +1$, respectively. Due to the $\pm$ sign in the construction of  $\Gamma_{g(x)}$  the resulting graph is  isomorphic to the one obtained already from $g_d(x) = 1 + x + \ldots + x^{s-1}$.
This graph is $C^{\pm 1}(p, rs, r)$ described in \cite{GP1}. The explicit construction 
arising from  $g_d(x) = 1 + x + \ldots + x^{s-1}$ is given by
$$M_{g(x)}=\begin{bmatrix} I | I | \ldots | I\end{bmatrix}
\in \ZZ_p^{r\times n},$$
where $I\in\ZZ_p^{r\times r}$ is the identity matrix.
The graph  $\Gamma_{g(x)}$ has  vertex set
$\ZZ_p^r \times \ZZ_n$ and adjacency relations
$(v,kr+j)\sim(v\pm e_{j+1},kr+j+1)$, where $0\leq k<s$, $0\leq j < n$, and $e_j\in \ZZ_p^r$ denotes the standard basis vector. 
Let $G$ be the largest group that lifts. Since  $\epsilon = 0$ we have $\rho \in G$, and $\tau_L = \tau_{\ZZ_n} \in B$ because of 
type $1$ reflexibility (see Remark~\ref{rem:tauL-reflextype}). Also, since $B$ is generated by the $\rho$-conjugates of $\tau_{[0,r]}$ we  have 
$G = \la  \tau_{[0,r]}, \rho, \sigma\ra \cong \Dih(n) \ltimes \ZZ_2^r$.

\bigskip
$\bullet$ \noindent{\sc Case 2}.
{\em Let $g_d(x)$ be reflexible of type $2$}: there exists $\lambda \in \ZZ_p^{*}$ (which must satisfy $\lambda^2 = (-1)^{s-1}$) such that  $\lambda \theta^{s-1-j} = (-1)^j \theta^j$ for all $j$. In particular, $\lambda \theta^{s- 1} = 1$, and so 
 $\lambda =  (-1)^{\epsilon}\theta$. This implies $\theta^2 = (-1)^{s-1}$. Moreover, $\theta^{2j} = (-1)^j$ for all $0 \leq j \leq s-1$. It follows that 
$(-1)^{(s-1)j} = (-1)^j$, that is,  $(-1)^{js} = 1$ for all $0 \leq j \leq s-1$. 
For $s$ odd we have  $(-1)^{j} = 1$ for all $0 \leq j \leq s-1$, so $s-1 = 0$. This case has already been discussed above.
Suppose that s is even. 
Then $\theta^2 = -1$, 
and we have a restriction on $p$, namely $p \equiv 1 (\mod\,4)$. Let $s = 2q$. From 
$\theta^{2q} = (-1)^{\epsilon}$ we obtain $(-1)^q = (-1)^{\epsilon}$. For $\epsilon = 0$ the integer $q$ must be even while for $\epsilon = 1$ it must be odd. In both cases 
we have $g_r(x) = 1 + \theta x + \ldots + \theta^{s-1} x^{s-1}$ and $g_d(x)$ arises either from the factorization of $x^s -1$ or $x^s +1$, respectively, where $\theta$ is an element of order $4$ in $\ZZ_p^{*}$. Since $\theta$ is defined up to $\pm$ sign we always obtain the same graph. This is the graph $C^{\pm \theta}(p,  2rq, r)$ described in \cite{GP1}.
Therefore, 
$$M_{g(x)}=\begin{bmatrix} \,I \mid \, \theta I\, | -I\, | -\theta I \mid I\,| \ldots | (-1)^{q+1}\theta I\,\end{bmatrix}
\in \ZZ_p^{r\times n},$$
where $I\in\ZZ_p^{r\times r}$ is the identity matrix. However, due to the $\pm$ sign
in the construction of  $\Gamma_{g(x)}$ we can as well take the matrix 
$$\begin{bmatrix} \,I \mid \theta I \mid I \mid  \theta I \mid  I \mid \ldots \mid \theta I\,\end{bmatrix}.$$
The vertex set is 
$\ZZ_p^r \times \ZZ_n$ and adjacency relations are
$(v,kr+j)\sim(v\pm e_{j+1},kr+j+1)$, for $k$ odd and 
$(v,kr+j)\sim(v\pm \theta e_{j+1},kr+j+1)$, for $k$ even, where 
$0\leq k<2q$, $0\leq j < r$, and $e_j\in \ZZ_p^r$ denotes the standard basis vector. 
Let $G$ be the largest group that lifts. If $q$ is even, then 
$\epsilon = 0$, so $\rho \in G$,  and $\tau_L = \Pi_{i=0}^{r-1} \tau_{[i,2r]}$ because of 
type $2$ reflexibility (see Remark~\ref{rem:tauL-reflextype}). Also, $B$ is generated by the $\rho$-conjugates of $\tau_{[0,r]}$.
Hence $G = \la  \tau_{[0,r]}, \rho, \sigma\tau_L\ra$. If $q$ is odd, then $\epsilon =1$, 
$\rho\tau_0 \in G$, and $\tau_L = \Pi_{i=0}^{r-1} \tau_{[i,2r]}$. Also, $B$ is generated by the $\rho\tau_0$-conjugates of $\tau_{[0,r]}$. Hence 
$G = \la \tau_{[0,r]}, \rho\tau_0, \sigma\tau_0\tau_L\ra$, where 
$\tau_L = \Pi_{i=0}^{r-1} \tau_{[i,2r]}$.

 \end{example}


\begin{small}

\end{small}

\end{document}